\documentclass[a4paper,egregdoesnotlikesansseriftitles,final]{scrartcl}
\usepackage[utf8]{inputenc}
\usepackage[british]{babel}
\usepackage[T1]{fontenc}
\usepackage{lmodern}
\usepackage{microtype}
\usepackage{amsfonts,amsthm,amssymb,amsmath}
\usepackage{mathtools}
\usepackage{letterswitharrows}
\usepackage{xcolor}
\usepackage{graphicx}
\usepackage[colorlinks,unicode,bookmarks,pdfusetitle,linkcolor={red!60!black},citecolor={green!60!black},urlcolor={blue!60!black}]{hyperref}
\usepackage[capitalize,nameinlink]{cleveref}
\crefname{enumi}{}{}
\crefname{property}{property}{properties}
\crefname{LEM}{Lemma}{the Lemmas}
\crefname{THM}{Theorem}{Theorems}
\usepackage{enumitem}
\usepackage{thmtools, thm-restate} 
\usepackage{cite}
\usepackage[abbrev]{amsrefs}
\usepackage{doi}
\renewcommand{\PrintDOI}[1]{\doi{#1}}

\newtheorem{THM}{Theorem}[section]
\newtheorem{LEM}[THM]{Lemma}
\newtheorem{COR}[THM]{Corollary}
\newtheorem{PROP}[THM]{Proposition}

\newtheorem{OBS}[THM]{Observation}
\theoremstyle{definition}
\newtheorem{EX}[THM]{Example}

\newcommand{\abs}[1]{\lvert#1\rvert}
\newcommand{\menge}[1]{\left\{#1\right\}}
\newcommand{\braces}[1]{\left(#1\right)}
\newcommand{\family}[1]{(\,#1\,)}
\renewcommand{\phi}{\varphi}
\newcommand{\Nbb}{\mathbb{N}}

\newcommand{\N}{\mathbb{N}}
\newcommand{\join}{\lor}
\newcommand{\meet}{\land}
\newcommand{\sub}{\subseteq}
\newcommand{\sm}{\smallsetminus}
\newcommand{\cP}{\mathcal{P}}
\newcommand{\cA}{\mathcal{A}}
\newcommand{\cB}{\mathcal{B}}
\newcommand{\cC}{\mathcal{C}}
\newcommand{\cY}{\mathcal{Y}}
\newcommand{\cK}{\mathcal{K}}
\newcommand{\cN}{\mathcal{N}}

\newcommand{\cU}{\mathcal{U}}
\newcommand{\cV}{\mathcal{V}}
\newcommand{\cX}{\mathcal{X}}

\def\invlim{\varprojlim\,}

\def\restricts{\!\restriction\!}
\def\proj{\restricts}

\title{Trees of tangles in infinite separation~systems}
\author{Christian Elbracht \and Jakob Kneip \and Maximilian Teegen}
\begin{document}
\maketitle
\begin{abstract}
We present infinite analogues of our splinter lemma from \cite{FiniteSplinters}.
From these we derive several tree-of-tangles-type theorems for infinite graphs and infinite abstract separation systems.
\end{abstract}

\section{Introduction}
One of the key concepts in graph minor theory is that of a tree-decomposition of a graph.
These allow us to analyse a large graph by dividing it into smaller subgraphs and, working with these, gain insight into the graph as a whole. 
Tangles, ever since their introduction by Robertson and Seymour~\cite{GMX}, have served as a dual object to tree-decompositions of low width in graphs. They have also been studied in their own right, as a way to express, indirectly, highly cohesive substructures such as $k$-blocks or large grid minors.

Over the years, this theory has been expanded and generalised; the branch of abstract tangle theory emerged.
There, the focus shifted from the study of graph minors to tangles as an abstract combinatorial object on an \emph{abstract separation system}, replacing the concrete separations of a graph.
These separation systems are just a poset which only expresses how separations lie in relation to each other, keeping little of the information provided by the underlying graph;
this notion is abstract enough to be applicable to settings other than graphs.
~\cites{AbstractSepSys,TangleTreeAbstract,TangleTreeGraphsMatroids,AbstractTangles,ProfilesNew,FiniteSplinters,ProfileDuality,ToTfromTTD}.

The two central theorems of abstract tangle theory are the \emph{tangle-tree duality theorem} and the \emph{tree-of-tangles theorem}. The former provides a tree-like structure dual to the existence of a tangle, while the latter exposes how tangles can be arranged in a tree structure by way of a nested set of distinguishing separations.
When applied to the vertex separations of a graph, such finite nested sets of separations can easily be converted back into a tree-decomposition in the original sense.

In \cite{FiniteSplinters} we introduced the `splinter lemma', a unified theorem which implies the known tree-of-tangles theorems for finite separations systems.
The merit of this theorem lies in the fact that, while it is strong enough to imply all these results, the proof of the theorem is simple, and its assumptions are easy to check. (See \cref{sec:terminology} for definitions.)

\begin{restatable}[Splinter lemma, \cite{FiniteSplinters}]{THM}{splinterThm}\label{thm:splinter}
    Let $ U $ be a universe of separations and $ (\cA_i)_{i\le n} $ a family of subsets of~$ U.$ If  $ (\cA_i)_{i\le n}$ splinters then we can pick an element $ a_i $ from each $ \cA_i $ so that $ \{a_1,\dots,a_n\} $ is nested.
\end{restatable}

\cref{thm:splinter}, in a sense, is yet another step in a series of abstractions in the theory of tangles: rather than working with the tangles themselves, it operates just on the collection of sets of separations distinguishing a given pair of these.

\cref{thm:splinter} is proved by induction: it finds a separation~$ a_i\in\cA_i $ which is nested with some element of every other~$ \cA_j $, and then proceeds inductively on the remaining~$ n{-}1 $ family members, restricted to those separations nested with~$ a_i $. This approach cannot deal with infinite families of sets, however.

In this paper we overcome these difficulties and present two different ways to obtain a version of \cref{thm:splinter} for infinite families of sets of separations. Both these versions are as abstract and therefore as widely applicable, as our original \cref{thm:splinter}; they differ in they way they overcome the difficulties of infinite sets of separations, and therefore in the assumptions required.

The first approach is to ensure that an inductive proof adopted from the finite case converges by imposing a closedness condition in a suitable topology: In \cref{sec:profinite} we present a result for profinite separation systems. These constitute a large class of infinite separation systems, including the systems of vertex separations of an infinite graph. In \cref{sec:profinite_app} we will, as an example, apply this theorem to precisely these graph separations.

However, to deduce the existing tangle-tree theorems in infinite graphs (see \cites{carmesinhalinconj,carmesin2020canonical}) we need a second approach which does not rely one these profinite separation systems, as not every tangle of an infinite graph is one of the corresponding profinite separation system. Following this approach, we obtain our second main theorem, which does imply these existing theorems (\cref{thm:Johannes,thm:jmb} below) and is therefore a truly infinite analogue of \cref{thm:splinter}.
It asks more of the sets~$\cA_i$ than the profinite version; in return the set of separations we obtain will be canonical, i.e., invariant under isomorphisms:

\begin{restatable}{THM}{splinterThinly}\label{thm:thinly}
If  $\family{\cA_i \mid i\in I}$ thinly splinters with respect to some reflexive symmetric relation $\sim$ on $\cA:=\bigcup_{i\in I}\cA_i$, then there is a set $N\subseteq \cA$ which meets every $\cA_i$ and is nested, i.e., $n_1\sim n_2$ for all $n_1,n_2\in N$.
Moreover, this set $N$ can be chosen invariant under isomorphisms:  if $ \phi $ is an isomorphism between $ (\cA,\sim)$ and $(\cA',\sim')$, 
then we have $ N(\family{\phi(\cA_i)\mid i\in I})=\phi(N(\family{\cA_i\mid i\in I})) $.
\end{restatable}

We prove this statement in \cref{sec:thinly}. Like \cref{thm:splinter}, the statement of this theorem is a bit technical, as we want it to be as widely applicable as possible. We then show the usefulness of this abstract theorem throughout \cref{sec:applications_thinly}, by deducing the existing theorems about distinguishing tangles in infinite graphs from it.

As a simple example, we start with applying it to tangles in locally finite graphs in \cref{sec:application_loc_fin}. This application is straightforward and demonstrates a prototypical application of \cref{thm:thinly}.

It is also possible to apply \cref{thm:thinly} to arbitrary infinite graphs, and we do so in \cref{sec:application_infinite}. This application uses another new, and interesting, shift of perspective: We cannot apply \cref{thm:thinly} directly to the sets of separations efficiently distinguishing two profiles since, in general, these do not splinter thinly. Instead we consider a slightly different set, namely, the set of separators, to which \cref{thm:thinly} does apply:
\begin{restatable}{THM}{thmNestedSeparators}
\label{cor:nested_separators}
Given a set of distinguishable robust regular profiles $\cP$ of a graph $G$ there exists a canonical nested set of separators efficiently distinguishing any pair of profiles in~$\cP$.
\end{restatable}
This theorem acts as an intermediary result between the existing results about tangles in arbitrary infinite graphs. On the one hand we can, waiving canonicity, transform the nested set of separators back into a nested set of separations, recovering the following result of Carmesin about distinguishing tangles in infinite graphs by way of a nested set of separations:
\begin{restatable}[\cite{carmesinhalinconj}*{Theorem 5.12}]{THM}{thmJohannes}\label{thm:Johannes} 
For any graph $G$, there is a nested set $N$ of separations that distinguishes efficiently any two robust principal profiles (that are not restrictions of one another).
\end{restatable}
This theorem is a cornerstone in Carmesin's proof that every infinite graph has a tree-decomposition displaying all its topological ends. For more about the relation between ends and tangles also see \cite{EndsAsTangles,EndsAndTangles}. We deduce \cref{thm:Johannes} from \cref{cor:nested_separators} in \cref{sub:johannes}.

On the other hand, if we want to keep canonicity, we can use \cref{cor:nested_separators} to deduce a result by Carmesin, Hamann and Miraftab \cite{carmesin2020canonical}. They construct a canonical object, which they call a tree of tree-decompositions, to distinguish the tangles:
\begin{restatable}[\cite{carmesin2020canonical}*{Remark 8.3}]{THM}{thmJMB}\label{thm:jmb}
 Let $G$ be a connected graph and $\cP$ a distinguishable set of principal robust profiles in $G$.
 There exists a canonical tree of tree-decompositions with the following properties:
 \begin{enumerate}[label=(\arabic*)]
     \item the tree of tree-decompositions distinguishes $\cP$ efficiently; \label{i:jmb1}
     \item if $t\in V(T)$ has level $k$, then $(T_t,\mathcal{V}_t)$ contains only separations of order $k$; \label{i:jmb2}
\item nodes $t$ at all levels have $|V(T_t)|$ neighbours on the next level and the graphs
    assigned to them are all torsos of $(T_t,\mathcal{V}_t)$. \label{i:jmb3}
 \end{enumerate}
\end{restatable}
We deduce \cref{thm:jmb} from \cref{cor:nested_separators} in \cref{sub:totd}.

\Cref{cor:nested_separators} is also an interesting result in its own right: the set of separators that it provides is a natural intermediate object between the non-canonical nested set of separations in \cref{thm:Johannes} and the canonical tree of tree-decompositions in \cref{thm:jmb}.
 
Moreover, proving \cref{thm:Johannes} or \cref{thm:jmb} by first proving \cref{cor:nested_separators} and then deducing them breaks up the proof nicely and is, in total, shorter than the original proofs from \cites{carmesinhalinconj,carmesin2020canonical}. 

\section{Terminology and basic facts} \label{sec:terminology}
All terminology for graphs in this paper is from \cite{DiestelBook16noEE}.
This paper builds extensively on the framework of \cites{FiniteSplinters,AbstractSepSys}.
A brief summary of their definitions follows, see \cites{AbstractSepSys,FiniteSplinters} for more in-depth exposition.

A poset $\vS$ together with an order-reversing involution $^*$ is called a \emph{separation system}, and its elements are \emph{(oriented) separations}. Given an oriented separation $\vs$ we shall denote its \emph{inverse} $\vs^*$ as $\sv$ and the pair $\{\vs,\sv\}$ as $s$. This $s$ is called the \emph{underlying unoriented separation} of $\vs$ and $\vs,\sv$ are its \emph{orientations}. We write $S$ for the set of all underlying unoriented separations of oriented separations in $\vS$. Informally we refer to either just as `separations' when the distinction is immaterial or the intended meaning is clear.

Two unoriented separation $r$ and $s$ are \emph{nested} if $ \vr\le\vs $ for suitable orientations of $ r $ and~$ s $. If $ r $ and $ s $ are not nested we say that they~\emph{cross}. Likewise we call oriented separations $ \vr $ and $ \vs $ \emph{nested} if $ r $ and $ s $ are nested, and \emph{crossing} otherwise.
A \emph{nested set of separations} is one whose elements are pairwise nested.

A separation $\vs$ is \emph{small} if $\vs\le\sv$, in which case we say that $\sv$ is \emph{cosmall}. A separation $\vs$ is \emph{trivial}, if there exists a separation $\vr\in \vS$ such that $r\neq s$ and both, $\vs\le\vr$ and $\vs\le\rv$. Note that every trivial separation is small.
An unoriented separation is said to be \emph{trivial} or \emph{small} if one of its orientations is.
A \emph{tree set} in a separation system $S$ is a nested set of separations which does not contain any trivial elements. It is \emph{regular} if it does not contain any small elements.

A separation system $\vU$ is a universe if there are join and meet operators $\join,\meet$ which turn the poset into a lattice. A map $ f\colon \vU\to\vU' $ between universes $ \vU $ and $ \vU' $ is an {\em homomorphism of universes} if $ f $ commutes with the involutions and the $ \join $ and $ \meet $ operations of $ \vU $ and $ \vU' $, i.e. if $ f(\sv)=f(\vs)^* $, $ f(\vr\join\vs)=f(\vr)\join f(\vs) $, and $ f(\vr\meet\vs)=f(\vr)\meet f(\vs) $ for all $ \vr,\vs\in\vU $. Clearly, if $ f\colon\vU\to\vU' $ is an homomorphism of universes, then $ \vr\le\vs $ for $ \vr,\vs\in\vU $ implies $ f(\vr)\le f(\vs) $ in~$ \vU' $.

An {\em isomorphism of universes} is a bijective homomorphism of universes whose inverse is also a homomorphism. Two universes are {\em isomorphic} if there is an isomorphism between them.

The \emph{corner separations} of two separations $s$ and $t$ are the four separations $\vs\join\vt$, $\vs\join\tv$, $\sv\join\vt$, $\sv\join\tv$, or their underlying unoriented separations. 

An \emph{order-function} on a universe $\vU$ is some function $\abs{\cdot}:\vU\to \Nbb_0$ such that ${\abs{\vs}=\abs{\sv}\eqqcolon\abs{s}}$ for every unoriented separation $s\in U$. Note that, unlike in the finite setting, in the infinite setting we require order-functions to take their values in $\Nbb_0$ rather than in~$\mathbb{R}_{\ge 0}$. In the finite case this does not make a big difference, as there were only finitely many orders of separations, which therefore could be scaled to be `essentially integer-valued'. Since such an argument is not possible in the infinite set-up, we need to require our order-function to be integer-valued in the first place.

An order-function is \emph{submodular} if for all $\vs,\vt\in \vU$ we have \[ \abs{\vs}+\abs{\vt}\ge \abs{\vs\join\vt}+\abs{\vs\meet\vt}\,. \]

An \emph{orientation} of some set $ S $ of unoriented separations is a set $O\subseteq \vS$ containing precisely one of $ \vs $ and $ \sv $ for each~$ s\in S $. An orientation $O$ is said to be \emph{consistent} if there are no $ \rv\le\vs $ with $ r\ne s $ and~$ \vr,\vs\in O $.

If $ O_1 $ and $ O_2 $ are two orientations of some (possibly distinct) sets of separations, a separation $ s $ \emph{distinguishes} $ O_1 $ and $ O_2 $ if $ \vs\in O_1 $ and $ \sv\in O_2 $ for a suitable orientation of~$ s $. If an order function is given, and $ s $ is of lowest possible order among all separations that distinguish $ O_1 $ and $ O_2 $, then $ s $ distinguishes them~\emph{efficiently.}

For a separation system $\vS\subseteq \vU$ inside some universe $\vU$, we say that a consistent orientation $O$ of $S$ is a \emph{profile} if 
it satisfies the profile property:
\[ \forall \,\vr,\vs\in P \colon (\rv\meet\sv)\notin P \tag{P}\label[property]{property:P} \]
A profile is said to be \emph{robust} if additionally:
\[\forall \vs\in P,\vt\in \vU:\text{ if }|\vs\join \vt|<|\vs|\text{ and }|\vs\join\tv|<|\vs|, \text{ then either }\vs\join\vt\in P\text{ or }\vs\join\tv\in P\]

Note that in \cite{ProfilesNew}, Diestel, Hundertmark and Lemanczyk also defined the much more technical but slightly weaker condition of a set of profiles being robust, which only requires the condition of robustness for some separations in $U$. Every set of distinguishable robust profiles is a \emph{robust set of profiles} in the sense of \cite{ProfilesNew}, and on the other hand, the reader familiar with this terminology may replace each appearance of a \emph{set of robust profiles} in this paper by a \emph{robust set of profiles} without changing the proofs.

If $U$ comes with some order function then, for some $k\in \Nbb\cup \{\aleph_0\}$, we write $S_k$ for the set of all separations in $U$ satisfying~$|s|<k$. A \emph{$k$-profile (in U)} is then a profile of $S_k$ and a \emph{profile in $U$} shall mean a $k$-profile in $U$ for some~$ k $. 

Given some universe $U$ and a collection $\cB\coloneqq\family{\cA_i\mid i\in I}$ of subsets $\cA_i$ of $U$, we say that $\cB$ \emph{splinters} if it satisfies, for any $i,j\in I$, the following property:
\[\forall s\in \cA_i, t\in \cA_j\colon s\in \cA_j\text{ or }t\in \cA_i\text{ or there exists a corner of }s\text{ and }t\text{ in }\cA_i\cup \cA_j.\]

One example of a separation system are the (vertex) separations of a graph:
Given some (infinite) graph $G$ let us consider the separations of $G$, these are all pairs $(A,B)$ of vertex sets $A,B$ such that $A\cup B=V(G)$ and there is no edge between $A\sm B$ and $B\sm A$.
For historical reasons and ease of notation the corresponding unoriented separations are represented as $\{A,B\}$ instead of $\{(A,B), (B,A)\}$, contrary to the notation in the abstract set-up.
These form a poset where $(A,B) \le (C,D)$ when $A \subseteq B$ and $C \supseteq D$. The function ${}^*: (A,B) \mapsto (B,A)$ is an order-reversing involution.
If $A\cap B$ is finite, the order of $(A,B)$ is defined as $\abs{A\cap B}$. 
Note that the set $S_{\aleph_0}$ of all separations of finite order is a universe of separations where $(A,B) \join (C,D) = (A\cup B, C\cap D)$ and $(A,B) \meet (C,D) = (A\cap B, C\cup D)$.
A \emph{profile in $G$} shall be a profile of $S_k\subseteq S_{\aleph_0}$ for some $k\in \Nbb\cup \{\aleph_0\}$.

A profile $ P $ in $ G $ is \emph{regular} if it does not contain any cosmall separation of $G$, i.e., it contains no separation of the form $(V(G), X)$. Note that, in graphs, the irregular profiles are not of large interest, since they always point towards either the empty set or a single non-cut-vertex. Formally, we can summarize this statement from \cite{ProfileDuality} as follows:
\begin{LEM}[{\cite{ProfileDuality}}]\label{lem:irregular}
  Let $G = (V,E)$ be a graph and $P$ an irregular profile in $G$ then either
  $G$ is connected and $P = \{(V, \emptyset)\}$
  or $G$ has a non-cutvertex $x \in V$ such that \[
      P = \{(A,B)\in\vS_2 \mid x\in B \text{ and } (A,B)\neq (\{x\},V)\}.
  \]
\end{LEM}
These irregular profiles are distinguished efficiently from each other and from all other profiles in $G$ by the set of separations \[
    \{\{V(G), \emptyset\}\} \cup \{ \{V(G), \{x\}\} \mid x\in V(G) \text{ and $x$ is not a cutvertex of $G$} \}.
\]
Every separation in this set is nested with all separations of $G$.
Hence, our efforts for applications in graphs will concentrate on regular profiles.

Given some set of vertices $X\subseteq V(G)$, we say that a connected component $C$ of $G-X$ is \emph{tight}, if $N(C)=X$.


For two vertices $x,y\in V(G)$ of a graph $G$, an \emph{$x$--$y$-separator of order $k$} is a vertex set $X\subseteq V(G)\sm \{x,y\}$ of size $k$ such that $x$ and $y$ lie in different components of $G-X$. We shall need the following basic fact about such separators in infinite graphs at various points throughout this paper.
\begin{LEM}[{\cite{halin1991lattices}*{2.4}}]\label{lem:finitelymany_seps}
Let $G$ be a graph, $u,v\in V(G)$ and $k\in \Nbb$. Then there are
only finitely many separators of size at most $k$ separating $u$ and $v$ minimally.
\end{LEM}

Another basic tool is the so-called `fish lemma':
\begin{LEM}[{\cite{ProfilesNew}*{Lemma~2.1}}]\label{lem:fish}
    Let $ U $ be a universe and $ r,s\in U $ two crossing separations. Every separation $ t $ that is nested with both $ r $ and $ s $ is also nested with all four corner separations of $ r $ and $ s $.
\end{LEM}

Additionally we shall use the following more general observation about separations nested with a corner separation:

\begin{LEM}\label{lem:crossnumb_graph}
	Let $ \vr $ and $ \vs $ be two separations. Every separation nested with one of~$ r $ or~$ s $ is also nested with at least one of~$ \vr\meet\vs $ and~$ \vr\join\vs $.
\end{LEM}

\begin{proof}
	Let $ t $ be a separation nested with, say,~$ r $. Then~$ t $ has an orientation~$ \vt $ with either~$ \vt\le\vr $ or~$ \vt\le\rv $. In the first case~$ t $ is nested with~$ \vr\join\vs $ by~$ \vt\le\vr\le(\vr\join\vs) $. In the latter case~$ t $ is nested with~$ \vr\meet\vs $ by~$ \vt\le\rv\le(\vr\meet\vs)^* $.
\end{proof}

The separations that distinguish a given pair of profiles exhibit a lattice-like structure:

\begin{LEM}\label{lem:lattice}
	Let $ \vU $ be a universe with a submodular order function and $ P $ and $ P' $ two profiles in~$ \vU $. If $ \vr,\vs\in P $ distinguish $ P $ and $ P' $ efficiently, then both $ \vr\join\vs $ and $ \vr\meet\vs $ also lie in $ P $ and distinguish $ P $ and $ P' $ efficiently.
\end{LEM}

\begin{proof}
	If one of $ \vr\join\vs $ and $ \vr\meet\vs $ has order at most $ \abs{r}=\abs{s} $, then that corner separation lies in $ P $ and distinguishes $ P $ and $ P' $ by their consistency and the~\cref{property:P}. The efficiency of $ r $ and $ s $ now implies that neither of the two considered corner separations can have order strictly lower than $ \abs{r} $. Therefore, by submodularity, both of them have order exactly $ \abs{r} $, which implies the claim.
\end{proof}

Moreover we shall need a way to transition between separations and tree-decompositions in graphs. Such a method already exists in finite graphs \cite{confing}. The ingredients of that proof together with the results of \cite{TreelikeSpaces} are all that is needed to show an analogous result for infinite graphs, which we shall present here.

Let us start by recalling the definition of a tree-decomposition. A \emph{tree-decomposition} of a $G$ is a pair $(T,\mathcal{V})$ of a tree $T$ together with a family $\mathcal{V} = (V_t)_{t \in T}$ of vertex sets $V_t\subseteq V(G)$ such that:
\begin{enumerate}[label={(T\arabic*)}]
    \item $V(G) = \bigcup_{t\in T}V_t$;\label{td:T1}
  \item Given $e\in E[G]$ there exists a $t\in T$ such that $e\subseteq V_t$;\label{td:T2}
  \item Given a path $P$ in $T$ from $t_1$ to $t_3$ and a vertex $t_2\in P$ we have $V_{t_1}\cap V_{t_3}\subseteq V_{t_2}$.\label{td:T3}
 \end{enumerate}
A separation $(A,B)$ is \emph{induced} by a tree-decomposition $(T,\mathcal{V})$ if and only if there exists an edge $tt'\in T$ such that for the components $T_t,T_{t'}$ of $T-tt'$ containing $t$ or $t'$ respectively, we have 
 \[(A,B)=\left(\bigcup_{t''\in T_t}V_{t''},\bigcup_{t''\in T_{t'}}V_{t''}\right)\,.\]

A \emph{chain of order type $\alpha$} or a \emph{$\alpha$-chain}, or a \emph{chain of length $\alpha$}, for some ordinal number $\alpha$ (we shall only use $\alpha\in \N$ or $\alpha=\omega$ or $\alpha=\omega+1$) is a collection $\{\vs_i\mid 0\le i<\alpha\}$ of oriented separations, such that $\vs_i<\vs_j$ whenever $i<j$. A set $N$ of unoriented separations \emph{contains} such a chain, if all the separations in that chain are orientations of separations from $N$.

Kneip and Gollin \cite{TreelikeSpaces} showed the following:
\begin{THM}[\cite{TreelikeSpaces}, Theorem 3.9]\label{lem:JakobPascal}
 Every regular tree set which does not contain a chain of order type $\omega+1$ is isomorphic to the edge tree set of a suitable tree.
\end{THM}
Here the \emph{edge tree set} of a tree $T=(V,E)$ is the set of separation obtained from the set of oriented edges $\vE$ by ordering them in the natural partial order (that is $vw\in \vE$ is smaller than $xy\in \vE$ if and only if the unique path from $v$ to $y$ meets $w$ and $x$).

We shall need the following lemma, whose proof is inspired by \cite{confing}.

\begin{LEM}\label{lem:TreesetTD}
	Let $G=(V,E)$ be an infinite graph and let $N\subseteq S_{\aleph_0}(G)$ be a regular tree set. If we have for any $\omega$-chain $(A_1,B_1)<(A_2,B_2)<\dots$ which is contained in $N$ that $\bigcap_{i\in \N}B_i=\emptyset$, then there exists a tree-decomposition $(T,\mathcal{V})$ of $G$ whose set of induced separations is $\vN$.
	
	Moreover this tree-decomposition can be chosen canonical: if~$ \phi\colon G\to G' $ is an isomorphism of graphs, then the tree-decomposition constructed for~$ \phi(N) $ in~$ G' $ is precisely the image under~$ \phi $ of the tree-decomposition constructed for~$ N $ in~$ G $.
\end{LEM}

\begin{proof}
	Let $T=(V,E)$ be the tree from~\cref{lem:JakobPascal}. Note that by~\cite{TreelikeSpaces}*{Theorem 3.9(iii)} any isomorphism between the edge tree sets of two distinct trees induces an isomorphism of the underlying trees.
	
	Let $\alpha$ be the isomorphism from the edge tree set of $T$ to~$ N $. Given some node $t\in T$ let us denote as $F_t$ the set of oriented separations
	\[F_t\coloneqq\{\alpha(s,t) \mid (s,t)\in\vE\}\,.\]
	We define the bags of our tree-decomposition as $V_t\coloneqq\bigcap_{(A,B)\in F_t}B$. Let us verify that~$ (T,\cV) $ with~$ \cV=\family{V_t}_{t\in T} $ is the desired tree-decomposition.
	
	For~\cref{td:T1} let~$ v\in V $ be given; we need to find a~$ t\in T $ with~$ v\in V_t $ . If~$ v\in A\cap B $ for some~$ (A,B)\in\vN $ then~$ v\in V_t $ for~$ t $ being either of the two end-vertices of the edge whose image under~$ \alpha $ is~$ (A,B) $. Otherwise~$ v $ induces an orientation~$ O $ of~$ E(T) $ by orienting each edge~$ \menge{x,y} $ of~$ T $ as~$ (x,y) $ if~$ v\in B\sm A $ for~$ (A,B)=\alpha(x,y) $.
	
	Observe that~$ O $ is consistent. If~$ O $ has a sink, that is, if there is a node~$ t $ of~$ T $ all of whose incident edges are oriented inwards by~$ O $, then~$ v\in V_t $ by definition of~$ O $. If~$ O $ does not have a sink then~$ O $ contains an~$ \omega $-chain. This is impossible though, since by definition of~$ O $ we would have~$ v\in\bigcap_{i\in\Nbb}B_i $, where~$ (A_i,B_i) $ is the image under~$ \alpha $ of the~$ i $-th element of that~$ \omega $-chain in~$ O $. Thus~\cref{td:T1} holds.
	
	The proof that~\cref{td:T2} holds can be carried out in much the same way due to the fact that every edge of~$ G $ is included in either~$ A $ or~$ B $ for each~$ (A,B)\in\vN $.
	
	Before we check that~\cref{td:T3} holds, let us show that~$ (T,\cV) $ indeed induces~$ N $. For this we need to show that if~$ (x,y) $ is an oriented edge of~$ T $ then
	\[ \alpha(x,y)=\braces{\bigcup_{z\in T_x}V_z\,,\,\bigcup_{z\in T_y}V_z}\,, \]
	where~$ T_x $ and~$ T_y $ are the components of~$ T-xy $ containing~$ x $ and~$ y $, respectively. So let~$ (x,y)\in\vE $ be given and~$ \alpha(x,y)=(A,B) $. Observe first that~$ A\cap B\sub V_x\cap V_y $ by definition. It thus suffices to show that~$ A\supseteq\bigcup_{z\in T_x}V_z $ and~$ B\supseteq\bigcup_{z\in T_y}V_z $ to establish the desired equality.
	
	To see this consider a vertex~$ v\in V_z $ for some~$ z\in T_x $. Let~$ \ve $ be the first edge of the unique~$ z $--$ x $-path in~$ T $ and let~$ \alpha(\ve)=(A',B') $. We have~$ \ve\le(x,y) $ by definition of an edge tree set, and hence~$ (A',B')\le(A,B) $ since~$ \alpha $ is an isomorphism. From this we know that~$ A'\sub A $. We further have~$ (B',A')\in F_z $ and thus, by definition of~$ V_z $, that~$ v\in A' $. This shows~$ v\in A $. The argument that~$ B\supseteq\bigcup_{z\in T_y}V_z $ is similar.
	
	Having established that~$ (T,\cV) $ indeed induces~$ N $, we can now deduce from this that~\cref{td:T3} holds: if~$V_{t_1}$ and~$V_{t_3}$ are two bags of~$ (T,\cV) $ which both contain some vertex~$v$, then~$v$ also needs to lie in the separator of every separation that is an image under~$ \alpha $ of an edge on the path~$P$ in~$ T $ from~$t_1$ to~$t_3$. Therefore~$v$ lies in every~$V_{t_2}$ with~$t_2\in P$.
\end{proof}

\section{The profinite splinter lemma}\label{sec:profinite}

In this section we establish an extension of the \hyperref[thm:splinter]{splinter lemma}
to a large class of infinite separation systems: the profinite universes. Informally, a separation system is profinite if it is determined entirely by its finite subsystems. The most prominent, and most important, example of such a universe of separations is the separation system of an infinite graph: two separations of an infinite graph are comparable precisely if all their restrictions to finite subgraphs are comparable. Moreover, a pair $ (A,B) $ of sets of vertices of an infinite graph $ G $ is a separation if and only if the restriction of $ (A,B) $ to each finite subgraph $ H $ of $ G $ is a separation of $ H $. We will make this relation between the separations of a profinite universe and their finite restrictions more formal now.

\subsection{Introduction to profinite universes}\label{sec:profintroduction}

For an in-depth introduction to profinite separation systems we refer the reader to~\cite{ProfiniteASS}, where this class of separation systems was first introduced. In this section we shall give only the definitions, terms, and tools for profinite universes relevant to our studies.

A {\em directed set} is a poset $ P $ 
in which every two elements have a common upper bound, i.e., in which there is an $ r\in P $ with $ p\le r $ and $ q\le r $ for all $ p,q\in P $. Given a directed set~$ P $, an {\em inverse system (of finite sets)} is a family $ \cX=\family{X_p\mid p\in P} $ of finite sets, together with maps $ f_{qp}\colon X_q\to X_p $ for all $ q>p $ that are {\em compatible} in the sense that $ f_{rp}=f_{qp}\circ f_{rq} $ for all $ r>q>p $. If every set $ X_p $ is a finite universe of separations $ \vU_p $, and the maps $ f_{qp} $ are homomorphisms of universes, then the family $ \cU=\family{\vU_p\mid p\in P} $ is an {\em inverse system (of universes of separations)}.

A {\em limit} of an inverse system $ \cX=\family{X_p\mid p\in P} $ is a compatible choice of one element $ x_p $ from each $ X_p $, that is, a family $ \family{x_p\mid p\in P} $ with $ x_p\in X_p $ and $ f_{qp}(x_q)=x_p $ for all~$ q>p $. The {\em inverse limit} $ \invlim\cX $ of $ \family{X_p\mid p\in P} $ is the set of all limits of $ \cX $. It is a well-known fact that every inverse system of non-empty finite sets has a limit (cf.\ \cite{ProfiniteASS}).

{\em Limits} and the {\em inverse limit} of an inverse system of universes are defined in the same way. Then the inverse limit $ \vU=\invlim\cU $ of an inverse system of universes $ \cU=\family{\vU_p\mid p\in P} $ is itself a universe of separations by defining involution, partial order, joins, and meets coordinate-wise. That is by, for $ \vr=\family{\vr_p\mid p\in P} $ and $ \vs=\family{\vs_p\mid p\in P} $, letting
\[ \sv\coloneqq\family{\sv_p\mid p\in P} \]
as well as
\[ \vr\join\vs\coloneqq\family{\vr_p\join\vs_p\mid p\in P} \]
and
\[ \vr\meet\vs\coloneqq\family{\vr_p\meet\vs_p\mid p\in P}, \]
with $ \vr\le\vs $ if and only if $ \vr_p\le\vs_p $ for all $ p\in P $. In particular the involution, joins and meets of limits of $ \cU $ are again limits of $ \cU $.

A universe of separations is then called {\em profinite} if it is isomorphic to an inverse limit of finite universes of separations. The most prominent example of a profinite universe of separations comes from infinite graphs:

\begin{EX}\label{ex:graphsprofinite}
	Let $ G=(V,E) $ be an infinite graph and $ \vU=\vU(G) $ the universe of all separations of $ G $, including those $(A,B)$ with $|A\cap B|=\infty$. Then $ \vU $ is profinite: let $ \cX $ be the set of all finite $ Z\sub V $. Then $ \cX $, ordered by inclusion, is a directed set. For $ Z\in\cX $ let $ \vU_Z $ be the universe of separations of $ G[Z] $. We define maps $ f_{ZY}\colon\vU_Z\to\vU_Y $ for $ Y\subset Z $ by letting $ f_{ZY} $ map a separation $ (A_Z,B_Z) $ of $ G[Z] $ to $ (A_Z\cap Y\,,\,B_Z\cap Y) $, which is easily seen to be a separation of $ G[Y] $. These maps are clearly compatible, and thus the family $ \cU\coloneqq\family{\vU_Z\mid Z\in\cX} $ is an inverse system of finite universes.
	
	Let us show that $ \vU $ is isomorphic to the inverse limit of~$ \cU $. For this observe that for every separation $ (A,B) $ of $ G $ the family of its restrictions $ \family{(A\cap Z\,,\,B\cap Z)\mid Z\in\cX} $ is a limit of $ \cU $, and the map $ f\colon\vU\to\invlim\cU $ given by mapping each $ (A,B)\in\vU(G) $ to this family $ \family{(A\cap Z\,,\,B\cap Z)\mid Z\in\cX} $ is a homomorphism of universes. Moreover $ f $ is clearly injective, and its inverse is a homomorphism as well. To see that $ f $ is an isomorphism between $ \vU $ and $ \invlim\cU $ it thus remains to show that $ f $ is surjective, that is, that every limit of $ \cU $ gives rise to a separation of $ G $.
	
	So let $ \vs=\family{(A_Z,B_Z)\mid Z\in\cX} $ be a limit of $ \cU $. Let $ A $ and $ B $ be the union of the sets $ A_Z $ and $ B_Z $, respectively, over all $ Z\in\cX $. We claim that $ (A,B) $ is a separation of $ G $. If so then $ f((A,B))=\vs $, showing that $ f $ is surjective.
	
	Note first that $ A\cup B=V $, since $ v\in A_{\menge{v}}\cup B_{\menge{v}} $ for every $ v\in V $. Suppose now that $ G $ contains an edge $ vw\in E $ with $ v\in A\sm B $ and $ w\in B\sm A $. Let $ Z\coloneqq\{v,w\} $ and consider the induced subgraph $ G[Z] $ of $ G $: by definition of $ (A,B) $ we have $ A_Z=\{v\} $ and $ B_Z=\{w\} $, but then $ vw $ is an edge of $ G[Z] $ between $ A_Z\sm B_Z $ and $ B_Z\sm A_Z $, contradicting the assumption that $ (A_Z,B_Z)\in\vU_Z $. Therefore $ (A,B) $ is indeed a separations of~$ G $.
\end{EX}

For the remainder of this section, let $ \cU=\family{\vU_p\mid p \in P} $ be an inverse system of universes and $ \vU $ its inverse limit.

Given an element $ \vs=\family{\vs_p\mid p\in P} $ of $ \vU $, we write $ (\vs\proj p)\coloneqq\vs_p $ for the projection of $ \vs $ to~$ \vU_p $. Likewise, for a set $ O\sub\vU $ the projection $ O\proj p $ to $ \vU_p $ is the set of all $ \vs\proj p $ with~$ \vs\in O $.

A family $ \family{N_p\mid p\in P} $ of finite subsets of the $ \vU_p $ is a {\em restriction} of $ \cU $ if $ f_{qp}(N_q)=N_p $ for all~$ q>p $. The inverse limit of such a restriction of $ \cU $ is a subset of~$ \vU=\invlim\cU $.

By equipping each $ \vU_p $ in $ \cU $ with the discrete topology, the inverse limit $ \vU=\invlim\cU $ becomes a topological space as a subspace of the product space~$ \prod_{p\in P}\vU_p $. Doing so makes the maps $ f_{qp} $ continuous, and it is easy to see that $ \vU $ is a closed subset of the product $ \prod_{p\in P}\vU_p $ and hence compact. In fact, the topology on $ \vU $ can be described in terms of the sets $ \vU_p $:

\begin{LEM}[\cite{ProfiniteASS}*{Lemma~4.1}]\label{lem:profclosed}\text{}
	\begin{enumerate}[label=\textup{(\roman*)},nosep,beginpenalty=10000]
	   \item The topological closure in~$\vU$ of a set $O\sub\vU$ is the set of all limits $\vs = \family{\vs_p\mid p\in P}$ with $\vs_p\in O\proj p$ for all~$p$.
	   \item A set $O\sub\vU$ is closed in~$\vU$ if and only if there are sets $O_p\sub\vU_p$, with $f_{qp}(O_q)\sub O_p$ whenever $p<q\in P$, such that $O = \invlim\, \family{O_p \mid p\in P}$.
	\end{enumerate}
\end{LEM}

We shall also use the following lemma, which is a re-formulation of Lemma~5.4 from~\cite{ProfiniteASS}:

\begin{LEM}\label{lem:profnested}
	A set $ N\sub \vU $ is nested if and only if $ (N\proj p)\sub \vU_p $ is nested for all $ p\in P $.
\end{LEM}

A consequence of~\cref{lem:profclosed} and \cref{lem:profnested} (which we will not use) is that the topological closure of a nested set is still nested.

Finally, let us show that two closed sets in $ \vU $ intersect if all of their projections do:

\begin{LEM}\label{lem:profintersect}
	If $ R,S\sub \vU $ are closed and $ (R\proj p)\cap(S\proj p) $ is non-empty for every $ p\in P $, then $ R\cap S $ is non-empty.
\end{LEM}

\begin{proof}
The family $ \family{(R\proj p)\cap(S\proj p)\mid p\in P} $ is an inverse system of non-empty finite subsets of $ \vU $, using as maps the restrictions of the maps $ f_{qp} $ of $ \family{\vU_p\mid p\in P} $. Since both $ R $ and $ S $ are closed in $ \vU $, every limit of this family lies in $ R\cap S $, which is therefore non-empty.
\end{proof}

\subsection{Statement and proof of the profinite splinter lemma}

Using the framework of profinite separation systems, we can now extend~\cref{thm:splinter} to infinite separation systems:

\begin{THM}\label{thm:profinite} 
Let $ \vU=\invlim\family{\vU_p\mid p\in P} $ be a profinite universe and $ \cB $ a family of non-empty closed subsets of~$ \vU $. If $ \cB $ splinters then there is a closed nested set $ N\sub \vU $ containing at least one element from each member of~$ \cB $. 
\end{THM}

\begin{proof}
For $ p\in P $ let $ \cB_p $ denote the projection $ \cB\proj p $ of $ \cB $ to $ \vU_p $, that is, the family of all~$ B\proj p $, where $ B $ is a member of $ \cB $. Then each projection $ \cB_p $ splinters in $ \vU_p $: consider two members $ B_p $ and $ B_p' $ of $ \cB_p $, with separations $ \vr_p\in B_p $ and $ \vs_p\in B_p' $. By definition of $ B_p $ and $ B_p' $ there are $ \vr\in B $ and $ \vs\in B' $ with $ (\vr\proj p)=\vr_p $ and $ (\vs\proj p)=\vs_p $. Since $ \cB $ splinters we have either that one of $ \vr $ or $ \vs $ lies in $ B\cap B' $, in which case its projection to $ \vU_p $ lies in $ B_p\cap B_p' $, or else some corner separation $ \vc $ of $ \vr $ and $ \vs $ lies in $ B\cup B' $. In the latter case $ (\vc\proj p)\in B_p\cup B_p' $ is a corner separation of $ \vr_p $ and $ \vr_s $, showing that $ \cB_p $ indeed splinters.

By the above observation and \cref{thm:splinter} applied to $ \cB_p $ and $ \vU_p $ there is a nested set $ N_p\sub \vU_p $ for every $ p\in P $ which contains an element of each member of $ \cB_p $. Let $ \cN_p $ be the set of all such nested sets~$ N_p\sub \vU_p $. Observe that if $ N_q\in\cN_q $ and $ q>p $ then $ f_{qp}(N_q)\sub\vU_p $ is a nested set meeting each member of $ \cB_p $ and hence lies in~$ \cN_p $. Therefore the family $ \family{\cN_p\mid p\in P} $ together with the maps mapping $ N_q\in\cN_q $ to $ f_{qp}(N_q)\in\cN_p $ for $ q>p $ is an inverse system of non-empty finite sets.

Let $ \family{N_p\mid p\in P} $ be a limit of $ \family{\cN_p\mid p\in P} $. Then this limit is a restriction of $ \cU $, and hence $ N\coloneqq\invlim\family{N_p\mid p\in P} $ is a subset of $ \vU $. In fact $ N $ is a closed nested subset of $ \vU $ by~\cref{lem:profclosed} and~\cref{lem:profnested}. To see that $ N $ contains an element of each member of $ \cB $, let $ B $ be a member of $ \cB $. Then $ (N\proj p)\cap (B\proj p) $ is non-empty for each $ p\in P $ since $ (N\proj p)=N_p\in\cN_p $, and thus by \cref{lem:profintersect} and the assumption that $ B $ is closed in $ \vU $ the sets $ N $ and $ B $ intersect.
\end{proof}

\section{Application of the profinite splinter lemma}\label{sec:profinite_app}

For this section, let $ G=(V,E) $ be a connected graph and $ \cX $ the set of finite subsets of~$ V $. As seen in~\cref{ex:graphsprofinite}, the universe $ \vU=\vU(G) $ of separations of $ G $ is profinite since it arises as the inverse limit of $ \family{\vU_Z\mid Z\in\cX} $, where $ \vU_Z $ denotes the universe of separations of~$ G[Z] $. Following the notation of~\cref{sec:profinite}, we write $ (A,B)\proj Z $ for the projection $ (A\cap Z\,,\,B\cap Z) $ of a separation $ (A,B)\in\vU $ to~$ \vU_Z $.

For $ k\in\N $ let $ \vS_k=\vS_k(G) $ be the separation system of all separations of order $ <k $ of~$ G $. Using~\cref{lem:profclosed}, it is easy to observe the following fact about these $ \vS_k $, which will be used throughout this section:

\begin{OBS}\label{obs:Skclosed}
	For every $ k\in\N $ the set $ \vS_k $ is a closed subset of~$ \vU $.
\end{OBS}

Our main goal in this section is to use~\cref{thm:profinite} to find a nested set of separations which efficiently distinguishes a large set of profiles of~$ G $. Concretely, we will be able to distinguish the set of all {\em regular bounded} profiles in~$ G $.
A profile $ P $ in $ G $ is {\em bounded} if $ P $ is a $ k $-profile of $ G $ for some $ k $ but is not a subset of any $ \aleph_0 $-profile of~$ G $.
Recall that a profile in $G$ is regular if it contains no separation of the form $(V(G), X)$.

The main result of this section, then, will be the following:

\begin{THM}\label{thm:profinite_johannes}
	Let $\cP$ be a set of robust regular bounded profiles in~$G$. Then there is a nested set $N$ of separations of $G$ which efficiently distinguishes all distinguishable profiles in~$\cP$.
\end{THM}

It can be shown~(\cite{EndsAndTangles}) that every $\aleph_0$-profile $P$ in $G$ corresponds to either an end of~$G$, or a so-called \emph{ultrafilter tangle}: an orientation which, for some $ X\in\cX $, induces a non-principal ultrafilter on the set of components of~$ G-X $. These ultrafilter tangles are studied extensively in \cite{EndsAndTangles}. Ends and ultafilter tangles both exhibit a different behaviour than profiles of finite graphs. Bounded profiles, on the other hand, {\em do} behave similarly to profiles of finite graphs, and consequently we shall be able to utilize~\cref{thm:profinite} to establish~\cref{thm:profinite_johannes}. The latter is a weakening of~\cref{thm:Johannes}, but our proof via~\cref{thm:profinite} will be significantly shorter and simpler than the proof of~\cref{thm:Johannes}.

For the remainder of this section let $ \cP $ be a set of robust regular bounded profiles in~$ G $. Given two profiles $P$ and $P'$ in $ \cP $ let $\cA_{P,P'}$ be the set of all separations of $ G $ that efficiently distinguish $P$ and~$P'$.

In order to deduce~\cref{thm:profinite_johannes} from~\cref{thm:profinite} we need to show that the family of the sets $ \cA_{P,P'} $ splinters, and that each $ \cA_{P,P'} $ is a closed subset of~$ \vU $. We will start by showing the latter:

\begin{PROP}\label{prop:closed}
	Let $ P $ and $ P' $ be distinguishable regular bounded profiles in~$ G $. Then $ \cA_{P,P'} $ is a closed subset of~$ \vU $.
\end{PROP}

In order to prove~\cref{prop:closed} we will first need to show a series of lemmas about how bounded profiles, and their efficient distinguishers, behave.
The first step is to show that a regular bounded profile, for every sufficiently small set $X$ of vertices, points towards a component of $G-X$. That is to say: bounded profiles do not behave like the ultrafilter tangles of~\cite{EndsAndTangles}.

\begin{LEM}\label{lem:aleph_0-tangles}
	Let $ X $ be a finite set of vertices and $ P $ a regular bounded profile of order at least $ \abs{X}+1 $ in~$ G $. Then there is a unique component $ C $ of $ G-X $ with~$ (V\sm C\,,\,C\cup X)\in P $.
\end{LEM}

\begin{proof}
Suppose that $P$ contains for every component $C$ of $G-X$ the separation $(C \cup X \,,\, V \sm C)$, we are going to construct an extension of $P$ to a profile of $S_{\aleph_0}$.

To determine the appropriate orientation of a separation $\menge{A,B} \in S_{\aleph_0}$, consider the components of $G-X$ and
let $C_A$ be the union of all those components which are contained in $A \sm B$. Likewise let $C_B$ be the union of all components contained in 
$B \sm A$ and $C_R$ the union of the remaining components, i.e., those which meet both $A$ and $B$. Since each of these needs to meet $A\cap B$ there are only finitely many.

By our assumption $P$ contains for every component $C$ of $G-X$ the separation $(C \cup X \,,\, V \sm C)$.
Since $C_R$ is a union of only finitely many components of $G-X$ and since $P$ is a profile, we have $(C_R \cup X \,,\, V \sm C_R) \in P$.

We now want to prove that one of $(V\sm C_B\,,\, C_B\cup N(C_B))$ and $(V\sm C_A\,,\, C_A\cup N(C_B))$ lies in $P$.
Indeed, if this is not the case then their respective inverses are in $P$, so the profile property gives us 
\[(C_B\cup C_A\cup N(C_A\cup C_B)\,,\, V\sm (C_A\cup C_B))\in P.\] 
This however would imply that the supremum \[
    (C_B\cup C_A\cup N(C_A\cup C_B)\,,\, V\sm (C_A\cup C_B)) \,\join\, (C_R \cup X \,,\, V \sm C_R) = (V, X)
\]
lies in $P$ by the profile property and the fact that this is a separation of order $\abs{X}$.
This contradicts the regularity of $P$.

This proves that one of $(V\sm C_B\,,\, C_B\cup N(C_B))$ and $(V\sm C_A\,,\, C_A\cup N(C_B))$ lies in $P$, and by consistency we cannot have both.
We may thus define an orientation $P'$ of $S_{\aleph_0}$ by declaring that $(A,B)$ shall be in $P'$ if and only if $(V\sm C_B\,,\, C_B\cup N(C_B))$ is in $P$. This orientation is consistent since $P$ is consistent and $(A,B) \le (V\sm C_B\,,\, C_B\cup N(C_B))$. Note that $P\subseteq P'$ and it only remains to show that $P'$ is a profile.

Given $(A,B), (C,D) \in P'$ we have, by definition, $(V\sm C_B\,,\, C_B\cup N(C_B))$ and\linebreak $(V\sm C_D\,,\, C_D\cup N(C_D))$ in $P$.
The profile property of $P$ then gives us \[
    (V\sm C_B\,,\, C_B\cup N(C_B)) \,\join\, (V\sm C_D\,,\, C_D\cup N(C_D)) \in P \subseteq P',
\]
so by the consistency of $P'$ we have $((A,B) \join (C, D))^* \notin P'$.
\end{proof}

Using~\cref{lem:aleph_0-tangles} we can take the next step towards showing that the sets $ \cA_{P,P'} $ are closed by showing that for every infinite chain in $ \cA_{P,P'} $ the sequence of the separators is eventually constant:

\begin{LEM}\label{lem:chain_seps}
	Let $ P $ and $ P' $ be distinguishable regular bounded profiles in $ G $ and consider and infinite increasing sequence $ {(A_1,B_1)\le(A_2,B_2)\le\dots} $ in~$ \cA_{P,P'} $. Then the sequence of separators $ (A_i\cap B_i)_{i\in\N} $ is eventually constant.
\end{LEM}

\begin{proof}
	By switching their roles if necessary we may assume that $ P' $ contains $ (B_1,A_1) $. Then, by consistency, $ (B_i,A_i)\in P' $ and consequently $ (A_i,B_i)\in P $ for every~$ i\in\N $.
	
	For $ i\in\N $ let us write $ X_i\coloneqq A_i\cap B_i $. By~\cref{lem:aleph_0-tangles} there is a unique component $ C_i $ of $ G-X_i $ with~$ (V\sm C_i\,,\,C_i\cup X_i)\in P $. Observe that, just like $ (A_i,B_i) $, the separation $ {(V\sm C_i\,,\,C_i\cup X_i)} $ distinguishes $ P $ and $ P' $ efficiently, and that $C_i \subseteq B_i$. This efficiency implies that~$ {N(C_i)=X_i} $. Observe further that the separations $ {(V\sm C_i\,,\,C_i\cup X_i)} $ form an increasing chain in~$ \cA_{P,P'} $ whose sequence of separators is~$ (A_i\cap B_i)_{i\ge\N} $. In fact $ C_i\supseteq C_j $ for~$ i\le j $. It thus suffices to show that the sequence of the $ C_i $ is eventually constant. So suppose for a contradiction that the sequence of the $ C_i $ is strictly decreasing.
	
	Let us first show that $ \bigcap_{i\in\N}C_i $ is empty. If not there is a vertex~$ v\in\bigcap_{i\in\N}C_i $. By applying~\cref{lem:aleph_0-tangles} to $ P' $ and each $ X_i $, we obtain components $ C_i' $ of $ G-X_i $ such that $ (V\sm C_i'\,,\,C_i'\cup X_i)\in P' $ for all~$ i\in\N $. Clearly $ C_i\ne C_i' $ and $ N(C_i')=X_i $ for all~$ i\in\N $. Fix any $ w\in C_1' $. Then $ w\in C_i' $ for every $ i\in\N $, and each $ X_i $ is a minimal $ v $-$ w $-separator in~$ G $. The latter contradicts~\cref{lem:finitelymany_seps} by the assumption that all~$ C_i $, and hence all~$ X_i $, are distinct.
	
	Thus $ \bigcap_{i\in\N}C_i $ is indeed empty. Let us define an orientation $ \tilde{P} $ of $ S_{\aleph_0} $ and show that $ \tilde{P} $ is a profile extending $ P $, obtaining a contradiction. Let $ \tilde{P} $ consist of all $ (A,B)\in\vS_{\aleph_0} $ for which there is a $ C_i $ with~$ C_i\sub B\sm A $. Since the $ C_i $ form a decreasing sequence of connected vertex sets with $ \bigcap_{i\in\N}C_i=\emptyset $, and $ A\cap B $ is finite, this defines an orientation of~$ S_{\aleph_0} $. Moreover it is easy to see that $ \tilde{P} $ is consistent and has the profile property.
	
	To obtain the desired contradiction it is thus left to check that~$ P\sub\tilde{P} $. However any $ (A,B)\in P $ with $ C_i\sub A\sm B $ for some $ i\in\N $ would be inconsistent with $ {(V\sm C_i\,,\,C_i\cup X_i)\in P} $. Therefore $ \tilde{P} $ is an extension of~$ P $, contrary to the assumption that $ P $ is bounded.
\end{proof}

Moreover, we can even show that the same statement also holds not only for chains, but even for the entire set $\cA_{P,P'}$:

\begin{LEM}\label{lem:finitely_different_separators}
	Let $ P $ and $ P' $ be distinguishable regular bounded profiles in $ G $. Then the separations $ (A,B)\in\cA_{P,P'} $ have only finitely many distinct separators $A\cap B$.
\end{LEM}

\begin{proof}
	Suppose for a contradiction that $ \cA_{P,P'} $ contains an infinite sequence of separations $ (A_1,B_1),(A_2,B_2),\dots $ whose separators $ A_i\cap B_i $ are pairwise distinct. We may assume without loss of generality that $ (A_i,B_i)\in P $ for every~$ i\in\N $. By~\cref{lem:lattice} $ P $ contains all finite joins of these separations. For each $ i\in\N $ let $ X_i $ be the separator of the supremum of $ (A_1,B_1) $ up to $ (A_i,B_i) $, and let $ C_i $ be the component of $ G-X_i $ with $ (V\sm C_i\,,\,C_i\cup X_i)\in P $ as given by~\cref{lem:aleph_0-tangles}. By~\cref{lem:chain_seps} the sequence of the~$ X_i $ is eventually constant, and therefore so is the sequence of the~$ C_i $ as $P$ is a profile. Let~$ C\coloneqq\bigcap_{i\in\N}C_i\ne\emptyset $.
	
	Analogously for $ P' $ let $ X_i' $ be the separator of the supremum of $ (B_1,A_1) $ up to $ (B_i,A_i) $ and let $ C_i' $ be the component of $ G-X_i' $ with $ (V\sm C_i'\,,\,C_i'\cup X_i')\in P' $. As before let~$ C'\coloneqq\bigcap_{i\in\N}C_i'\ne\emptyset $. 
	
	Since $ C_1 $ and $ C_1' $ are disjoint so are $ C $ and~$ C' $. Fix vertices $ v\in C $ and $ w\in C' $. We claim that every separator $ A_i\cap B_i $ is a minimal $ v $-$ w $-separator in $ G $, contradicting the assertion of~\cref{lem:finitelymany_seps}. To see this, consider $ A_i\cap B_i $ for some $ i\in\N $. Let $ \tilde{C}_i $ and $ \tilde{C}_i' $ be the components of $ G-(A_i\cap B_i) $ obtained by applying~\cref{lem:aleph_0-tangles} with $ P $ and $ P' $, respectively. Then $ C\sub C_i\sub\tilde{C}_i $ and likewise $ C'\sub C_i'\sub\tilde{C}_i' $, giving $ v\in\tilde{C}_i $ and~$ w\in\tilde{C}_i' $. Moreover both $ \tilde{C}_i $ and $ \tilde{C}_i' $ have all of $ A_i\cap B_i $ as their neighbourhood as $(A_i,B_i)$ efficiently distinguishes $P$ and $P'$, and hence $ A_i\cap B_i $ is indeed a minimal $ v $-$ w $-separator in~$ G $.
\end{proof}

We are now ready to prove~\cref{prop:closed}, i.e. that the sets $ \cA_{P,P'} $ are closed subsets of~$ \vU $:

\begin{proof}[Proof of \cref{prop:closed}]
	Let two distinguishable regular bounded profiles $ P $ and $ P' $ in $ G $ be given. By~\cref{lem:finitely_different_separators} only finitely many sets, say $ X_1,\dots,X_n $, appear as separators of separations in~$ \cA_{P,P'} $. For each $ X_i $ let $ C_i $ and $ C_i' $ be the two components of $ G-X_i $ given by applying~\cref{lem:aleph_0-tangles} to $ X_i $ for $ P $ and $ P' $, respectively.
	
	We are now able to give a complete description of the set~$ \cA_{P,P'} $: it is easy to check that a separation $ (A,B)\in\vS_{\aleph_0} $ distinguishes $ P $ and $ P' $ efficiently if and only if $ A\cap B=X_i $ for some $ i $ with one of $ C_i $ and $ C_i' $ being a subset of $ A $ and the other a subset of~$ B $.
	
	For each $ X_i $, the set of all $ (A,B)\in\vU $ with $ A\cap B=X_i $ as well as $ C_i\sub A $ and $ C_i'\sub B $ is closed by~\cref{lem:profclosed}. Likewise the set of all $ (A,B)\in\vU $ with separator $ X_i $ as well as $ C_i'\sub A $ and $ C_i\sub B $ is closed, too. Therefore $ \cA_{P,P'} $ is the union of finitely many closed subsets of $ \vU $ and hence closed.
\end{proof}

Having established that the sets $ \cA_{P,P'} $ are closed in $ \vU $, it thus remains for us to verify that the family of the sets $ \cA_{P,P'} $ splinters in order to deduce~\cref{thm:profinite_johannes} from~\cref{thm:profinite}. Since we shall need a slightly stronger property than splintering at a later point in~\cref{sec:applications_thinly}, we will prove this stronger assertion here. In particular we will not make use of the assumptions that the profiles in $ \cP $ are regular and bounded.

To show that the sets $ \cA_{P,P'} $ splinter, we need to show that for all $ (A,B)\in\cA_{P,P'} $ and $ (C,D)\in\cA_{Q,Q'} $, either some corner separation of $ (A,B) $ and $ (C,D) $ lies in $ \cA_{P,P'} $ or in $ \cA_{Q,Q'} $, or one of $ (A,B) $ and $ (C,D) $ lies in both $ \cA_{P,P'} $ and $ \cA_{Q,Q'} $. In fact we will show that the first option always occurs.

We will split our proof of this into two separate lemmas, distinguishing the cases of equal and of distinct order of $ (A,B) $ and $ (C,D) $.

Let us first deal with the case that $ (A,B) $ is of strictly lower order than $ (C,D) $. In this case we can say precisely which of $ \cA_{P,P'} $ and $ \cA_{Q,Q'} $ will contain a corner separation of $ (A,B) $ and $ (C,D) $:

\begin{LEM}\label{lem:findcorner1}
	Let $ (A,B)\in\cA_{P,P'} $ and $ (C,D)\in\cA_{Q,Q'} $ with $ \abs{(A,B)}<\abs{(C,D)} $. Then some corner separation of $ (A,B) $ and $ (C,D) $ lies in~$ \cA_{Q,Q'} $.
\end{LEM}

\begin{proof}
Since $|(A,B)|<|(C,D)|$ it follows that both $Q$ and $Q'$ orient $\{A,B\}$ the same, say~${(A,B)\in Q\cap Q'}$. If ${|(A,B) \join (C,D)|\le |(C,D)|}$ or ${|(A,B) \join (D,C)|\le |(C,D)|}$, it follows that this corner separation efficiently distinguishes $Q$ and $Q'$ by~\cref{lem:fish}, so suppose that this is not the case. Then submodularity implies that ${|(B,A) \join (C,D)|< |(A,B)|}$ and ${|(B,A) \join (D,C)|< |(A,B)|}$, which in turn contradicts the efficiency of $(A,B)$, since one of $(B,A) \join (C,D)$ and $(B,A) \join (D,C)$ would also distinguish the two robust profiles $ P $ and~$ P' $.
\end{proof}

For separations $ r $ and $ s $ the corner separations given by $ \vr\join\vs $ and $ \rv\join\sv $ (as well as their underlying unoriented separations) are referred to as \emph{opposite} corner separations.

The second case is that $ (A,B) $ and $ (C,D) $ are of equal order. Here we can show that there are two opposite corner separations of $ (A,B) $ and $ (C,D) $ that lie in $ \cA_{P,P'} $ or in~$ \cA_{Q,Q'} $:

\begin{LEM}\label{lem:findcorner2}
	Let $ (A,B)\in\cA_{P,P'} $ and $ (C,D)\in\cA_{Q,Q'} $ with $ \abs{(A,B)}=\abs{(C,D)} $. Then there is either a pair of two opposite corner separations of $ (A,B) $ and $ (C,D) $ with one element in $ \cA_{P,P'} $ and one in $ \cA_{Q,Q'} $, or else there are two pairs of opposite corner separations of $ (A,B) $ and $ (C,D) $, the first with both elements in $ \cA_{P,P'} $ and the second with both elements in~$ \cA_{Q,Q'} $.
\end{LEM}

\begin{proof}
	From $|(A,B)|=|(C,D)|$ it follows that $ P $ and $ P' $ both orient $ \{C,D\} $, and likewise that $ Q $ and $ Q' $ both orient~$ \{A,B\} $.
	
	Let us first treat the case that one of $ P $ and $ P' $ orients both $ \{A,B\} $ and $ \{C,D\} $ in the same way as one of $ Q $ and~$ Q' $ does. So suppose that, say, both $ P $ and $ Q $ contain $ (A,B) $ as well as~$ (C,D) $. 
	
	If $ P' $ contains $ (D,C) $, then $ (C,D)\in\cA_{P,P'} $ and~\cref{lem:lattice} gives~$ {(A,B)\join(C,D)\in\cA_{P,P'}} $ and~$ {(B,A)\join(D,C)\in\cA_{P,P'}} $.  Thus by~\cref{property:P} we also have $ (A,B)\join(C,D)\in\cA_{Q,Q'} $, producing the desired pair of opposite corner separations. If $ Q' $ contains $ (B,A) $ we argue analogously.
	
	So suppose that $ (C,D)\in P' $ and~$ (A,B)\in Q' $. Then $ (B,A)\join(C,D)\in P' $ and $ (A,B)\join(D,C)\in Q' $ by the profile property, since by submodularity and the efficiency of $ (A,B) $ and $ (C,D) $ both of these corner separations have order exactly~$ \abs{(A,B)} $. These two separations, then, are opposite corner separations of $ (A,B) $ and $ (C,D) $ with the first lying in $ \cA_{P,P'} $ and the second lying in~$ \cA_{Q,Q'} $.
	
	The remaining case is that no two of the four profiles agree in their orientation of $ \{A,B\} $ and~$ \{C,D\} $. But then both of $ (A,B) $ and $ (C,D) $ lie in $ \cA_{P,P'} $ as well as in~$ \cA_{Q,Q'} $, and the existence of two pairs of opposite corner separations, one with both elements in $ \cA_{P,P'} $ and one with both in $ \cA_{Q,Q'} $, follows from~\cref{lem:lattice} and the disagreement of the four profiles on $ \{A,B\} $ and~$ \{C,D\} $.
\end{proof}

We now have all the ingredients necessary for a proof of~\cref{thm:profinite_johannes}:

\begin{proof}[Proof of \cref{thm:profinite_johannes}]
By \cref{prop:closed}, we can apply \cref{thm:profinite}. Thus we only need to show that the collection of these sets $\cA_{P,P'}$ splinters. However, this follows from \cref{lem:findcorner1} and \cref{lem:findcorner2}.
\end{proof}

We remark that even in locally finite graphs it is not generally possible to find a \emph{tree-decomposition} which efficiently distinguishes all the distinguishable robust regular bounded profiles, as witnessed by the following example:
\begin{EX}\label{ex:notd}

\begin{figure}[ht]
    \includegraphics[width=\linewidth]{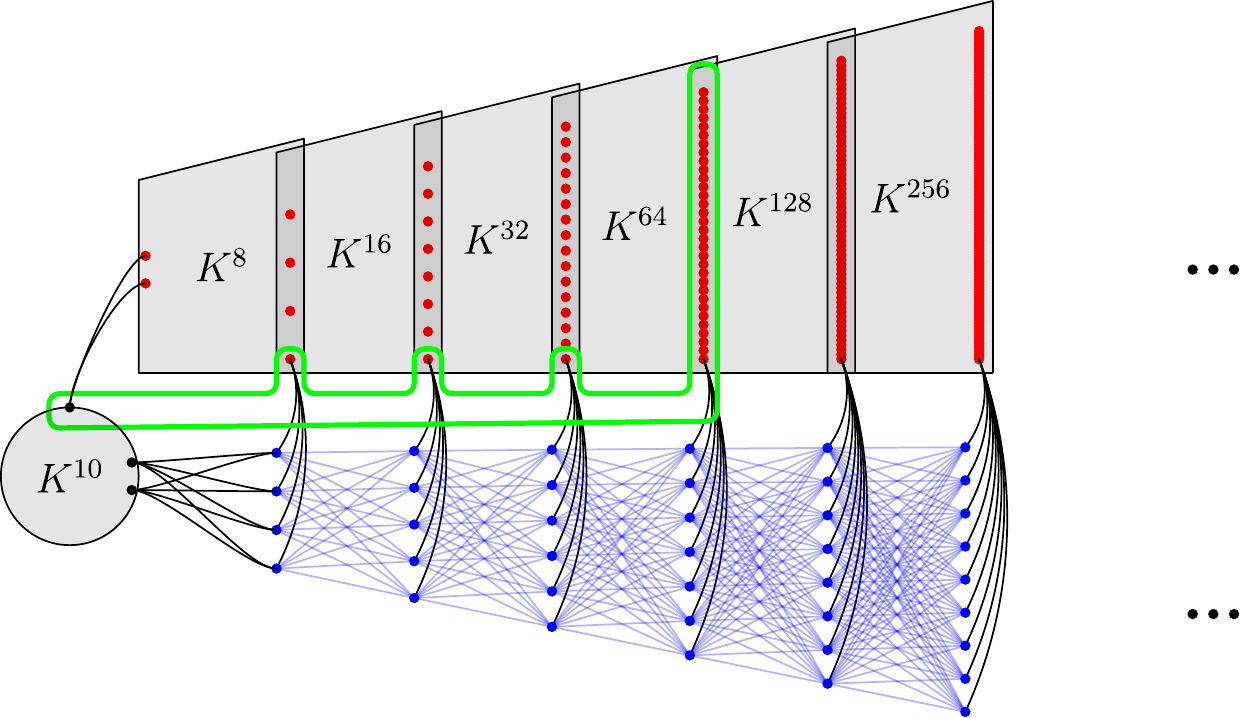}
    \caption{A locally finite graph where no tree-decomposition distinguishes all the robust regular bounded profiles efficiently. The green separator is the one of the only separation which efficiently distinguishes the profile induced by the~$K^{64}$ from the profile induced by the~$K^{128}$.}
    \label{fig:no_efficient_td}
\end{figure}
Consider the graph displayed in \cref{fig:no_efficient_td}. This graph is constructed as follows: for every $n\in \Nbb$ pick a copy of $K^{2^{n+2}}$ together with $n+3$ vertices $w_1^n,\dots,w_{n+3}^n$. Pick $2^{n}$ vertices of the $K^{2^{n+2}}$  and call them $u_1^n,\dots,u_{2^n}^n$. Additionally, pick $2^{n+1}$ vertices from  $K^{2^{n+2}}$, disjoint from the set of $u_i^n$,  and call them $v_1^n,\dots,v_{2^{n+1}}^n$. Now identify $u_i^{n+1}$ with $v_i^n$ and add edges between every $w_i^n$ and every $w_j^{n+1}$ as well as between $w_i^n$ and~$v_1^n=u_1^{n+1}$.

Finally we pick one copy of $K^{10}$ and join one vertex $v_1^0$ of this $K^{10}$ to $u_1^1$ and $u_1^2$. Additionally we pick two vertices $w_1^0,w_2^0$ which are distinct from $v_1^0$  from this $K^{10}$ and add an edge between each $w_i^0$ and each $w_j^1$.

Now each of the chosen $K^{2^{n+2}}$ induces a robust profile $P_n$ of order~$\frac{2}{3}\cdot2^{n+1}$ which obviously is regular and bounded. The only separation which efficiently distinguishes $P_n$ and $P_{n+1}$ is the separation $s_n$ with separator~${\{v_1^i\mid i<n\}\cup\{u_j^{n+1}\}}$.

Additionally, the $K^{10}$ induces a robust profile $P_0$ of order $4$. However the only separation that efficiently distinguishes  $P_0$ and $P_1$ has the separator~$\{v_1^0,w_1^0,w_2^0\}$. But these separations $s_1,s_2,...,$ and $s_0$ can be oriented such as to form a chain of order type~$ \omega+1 $. This chain witnesses that there cannot be a tree-decomposition which distinguishes all regular bounded profiles efficiently: the separations given by such a tree-decomposition would have to contain this chain of order type~$ \omega+1 $ which is not possible as every chain in the edge tree set of a tree has length at most $\omega$, cf.\ \cref{lem:TreesetTD}.
\end{EX}

\section{The thin splinter lemma}\label{sec:thinly}
In this section we take a different approach to generalising the finite splinter lemma into an infinite setting. Unlike \cref{thm:profinite}, the result we are going to prove does not require our universe to be `closed'. Instead we will require that the separations involved do not, in a sense, cross too badly in that they cross only finitely many separations of lower order.

This will allow us to choose separations that minimise the number of separations crossing them, an idea which also appeared in Carmesin's original proof of \cref{thm:Johannes} in \cite{carmesinhalinconj}, as well as in \cite{carmesin2020canonical} and our proof of the canonical spinter theorem for finite separation systems in \cite{FiniteSplinters}. However, our theorem here applies to a more general setting and will allow us directly to deduce Carmesin's theorem for locally finite graphs. 

In order to also be able to deduce the full \cref{thm:Johannes} for arbitrary graphs, we will state our theorem in more generality here: not as a theorem about nestedness and separations, but as a theorem about a general nestedness-like relation. This allows us to apply the theorem in \cref{sec:application_infinite} not to separations directly, where it would fail, but to substitute separators as a proxy giving our \cref{cor:nested_separators}. From this result we will retrieve the separations for our proof of \cref{thm:Johannes} in \cref{sub:johannes}, but we will also build from this a tree of tree-decompositions to deduce \cref{thm:jmb} in \cref{sub:totd}.

The statement of our \cref{thm:thinly} is also inspired by our canonical splinter lemma for the finite setting in \cite{FiniteSplinters}, and it will too result in a \emph{canonical} nested set, a set which is invariant under isomorphisms.

So let $ \cA $ be some set and $ \sim $ a reflexive and symmetric binary relation on~$ \cA $. In analogy to our terminology for separation systems, we say that two elements $ a $ and $ b $ of $ \cA $ are \emph{nested} if~$ a\sim b $. Elements of $ \cA $ that are not nested \emph{cross}. As usual, a subset of $ \cA $ is nested if all of its elements are pairwise nested, and a single element is nested with a set $ N $ if it is nested with every element of~$ N $.

In an abuse of notation, given elements $ a $ and $ b $ of $ \cA $, we call $ c\in\cA $ a \emph{corner} of $ a $ and $ b $ if every element of~$ \cA $ crossing $ c $ also crosses one of $ a $ and~$ b $. Observe that with this definition corners of elements of $ \cA $ exhibit the same behaviour as was asserted by~\cref{lem:fish} for corner separations. However, in contrast to the terminology of separation systems, we do not insist here that a corner of $ a $ and $ b $ is itself nested with both $ a $ and~$ b $. This distinction will become relevant in \cref{sec:application_infinite}.

Now let $ \family{\cA_i\mid i\in I} $ be a family of non-empty subsets of $ \cA $ and $ \abs{\;\;}\colon I\to\N_0 $ some function, where $ I $ is a possibly infinite index set. We shall think of $ \abs{i} $ as the order of the elements of~$ \cA_i $. For an $ a\in\cA $ and $ k\in\N_0 $ the $ k $-\emph{crossing number} of~$ a $ is the number of elements of~$ \cA $ that cross~$ a $ and lie in some~$ \cA_i $ with~$ \abs{i}=k $. This $ k $-crossing number is either a natural number or infinity. The family $ \family{\cA_i\mid i\in I} $ \emph{thinly splinters} if it satisfies the following three properties:

\begin{enumerate}
	\item For every $ i\in I $ all elements of $ \cA_i $ have finite $ k $-crossing number for all~$ k\le\abs{i} $. \label[property]{property:fin_cn}
	\item If $ a_i\in\cA_i $ and $ a_j\in\cA_j $ cross with $ \abs{i}<\abs{j} $, then $ \cA_j $ contains some corner of~$ a_i $ and~$ a_j $ that is nested with~$ a_i $. \label[property]{property:fish}
	\item If $ a_i\in\cA_i $ and $ a_j\in\cA_j $ cross with $ \abs{i}=\abs{j}=k \in \N_0 $, then either $ \cA_i $ contains a corner of $ a_i $ and $ a_j $ with strictly lower $ k $-crossing number than~$ a_i $, or else $ \cA_j $ contains a corner of $ a_i $ and $ a_j $ with strictly lower~$ k $-crossing number than~$ a_j $. \label[property]{property:strong_submodular}
\end{enumerate}


We are now ready to state and prove the main result of this section:

\splinterThinly*

\begin{proof}
	We shall construct inductively, for each $ k\in\N_0 $, a nested set $ N_k\sub\cA $ extending $ N_{k-1} $ and meeting every $ \cA_i $ with $ \abs{i}\le k $, such that the choice of $ N_k $ is invariant under isomorphisms. The desired nested set $ N $ will then be the union of all these sets~$ N_k $.
	
	We set~$ N_{-1}\coloneqq\emptyset $. Suppose that for some $ k\in\N_0 $ we have already constructed a nested set~$ N_{k-1} $ such that $ N_{k-1} $ is canonical and meets every $ \cA_i $ with $ \abs{i}\le k-1 $. We shall construct a canonical nested set $ N_k\supseteq N_{k-1} $ that meets every $ \cA_i $ with~$ \abs{i}\le k $.
	
	Let $ N^+_k $ be the set consisting of the following: for every $ i\in I $ with $ \abs{i}=k $, among those elements of $ \cA_i $ that are nested with $ N_{k-1} $, those of minimum $ k $-crossing number. We claim that $ N_k\coloneqq N_{k-1}\cup N_k^+ $ is as desired.
	
	Since the choice of $ N^+_k $ is invariant under isomorphisms, and $ N_{k-1} $ is canonical by assumption, $ N_k $ is clearly canonical as well. It thus remains to show that~$ N_k $ meets every~$ \cA_i $ with~$ \abs{i}=k $, and that the set $ N_k $ is nested.
	
	To see that the former is true, let $ i\in I $ with $ \abs{i}=k $ be given. It suffices to show that~$ \cA_i $ contains some element that is nested with~$ N_{k-1} $. If $ \cA_i $ already meets $ N_{k-1} $ there is nothing to show, so suppose that it does not. By~\cref{property:fin_cn} every element of~$ \cA_i $ crosses only finitely many elements of~$ N_{k-1} $; pick an $ a_i\in\cA_i $ that crosses as few as possible. Suppose for a contradiction that $ a_i $ crosses some element of $ N_{k-1} $, that is, some $ a_j\in\cA_j $ with~$ \abs{j}<\abs{i} $. But then, by~\cref{property:fish}, $ \cA_i $ contains a corner of $ a_i $ and~$ a_j $ that is nested with~$ a_j $. This element of $ \cA_i $ does not cross $ a_j $ and therefore, by virtue of being a corner of $ a_i $ and $ a_j $, crosses fewer elements of $ N_{k-1} $ than $ a_i $ does, contrary to the choice of~$ a_i $. Therefore $ N_k $ indeed contains an element of each $ \cA_i $ with~$ \abs{i}\le k $.
	
	Let us now show that $ N_k $ is nested. Since $ N_{k-1} $ is a nested set by assumption, and every element of $ N^+_k $ is nested with $ N_{k-1} $, we only need to show that the set $ N^+_k $ itself is nested. So suppose that some two elements of $ N^+_k $ cross. These two elements then are some~$ a_i\in\cA_i $ and~$ a_j\in\cA_j $ with~$ \abs{i}=\abs{j}=k $. But now~\cref{property:strong_submodular} asserts that one of $ \cA_i $ and $ \cA_j $ contains a corner of $ a_i $ and $ a_j $ with a strictly lower $ k $-crossing number than the corresponding element~$ a_i $ or~$ a_j $. Since both $ a_i $ and $ a_j $ are nested with $ N_{k-1} $ their corner is nested with $ N_{k-1} $ as well, and hence contradicts the choice of~$ a_i $ or~$ a_j $ for~$ N^+_k $.
\end{proof}

\section{Applications of the thin splinter lemma}\label{sec:applications_thinly}

In this section we are going to apply \cref{thm:thinly} to infinite graphs. The application to locally finite graphs in \cref{sec:application_loc_fin} will be a straightforward application to a universe of separations, whereas in \cref{sec:application_infinite} we are going to use a more involved argument.

For either case we will utilise the fact that separations which efficiently distinguish two regular profiles are tight. Recall that for a set $ X\sub V $ a component $ C $ of $ G-X $ is \emph{tight} if~$ N(C)=X $. We say that a separation $ (A,B) $ of $ G $ is \emph{tight} if for $ X\coloneqq A\cap B $ each of $ A\sm B $ and $ B\sm A $ contains some tight component of~$ G-X $.

\begin{LEM}\label{lem:distinguisher_tight}
 Let $P,P'$ be two distinct regular profiles in an arbitrary graph $G$. If $(A,B)$ is a separation of finite order that efficiently distinguishes $P$ and $P'$, then $(A,B)$ is tight.
\end{LEM}
\begin{proof} Let $(A,B)\in P$, $(B,A)\in P'$. 

    Suppose for a contradiction that $B\sm A$ does not contain a tight component of $G-(A\cap B)$. Let $Y_1,\dots Y_m$ be an enumeration of all proper subsets of $A\cap B$. For every $Y_l$ let $\cC_l$ be the set of components of $G-(A\cap B)$ in $B$ with neighbourhood exactly $Y_l$. By consistency of $P'$ we have $(\bigcup \cC_l\cup Y_l,V\sm \bigcup \cC_l)\in P'$. Since moreover $(A,B)$ efficiently distinguishes $P$ from $P'$ and $|Y_l| < |A\cap B|$ we know that $(\bigcup \cC_l\cup Y_l,V\sm \bigcup \cC_l)\in P$ as well. Moreover, $(A\cap B,V)\in P$ since $P$ is regular. Thus, by an inductive application of the profile \cref{property:P} we have that for every $l$
    \[(A\cap B,V)\vee(\bigcup \cC_1\cup Y_1,V\sm \bigcup \cC_1)\vee\dots\vee (\bigcup \cC_l\cup Y_l,V\sm \bigcup \cC_l)\in P.\]
However, for $l=m$ this contradicts the assumption since
    \[(A\cap B,V)\vee(\bigcup \cC_1\cup Y_1,V\sm \bigcup \cC_1)\vee\dots\vee (\bigcup \cC_m\cup Y_m,V\sm \bigcup \cC_m))=(B,A)\notin P\,. \]
\end{proof}

\subsection{Locally finite graphs}\label{sec:application_loc_fin}
In this section we apply \cref{thm:thinly} to the set of separations of a locally finite graph, which will result in a canonical nested set of separations efficiently distinguishing any two distinguishable regular profiles in $G$. The proof of this theorem will be a straightforward application of \cref{thm:thinly} to sets $\cA_{P,P'}$ of separations efficiently distinguishing two profiles in $G$. Following the strategy of this proof, one might be able to obtain similar results for other infinite separation systems, e.g., in a matroid.

So let $G=(V,E)$ be a locally finite connected graph and $\cP$ a set of robust regular profiles in~$G$.

Let $I$ be the set of pairs of distinguishable profiles in~$\cP$. For each pair $ P $ and $ P' $ of distinguishable profiles in~$ \cP $ let~$ \cA_{P,P'} $ be the set of all separations of $ G $ that distinguish~$ P $ and~$ P' $ efficiently. Observe that by definition all separations in $ \cA_{P,P'} $ are of the same order; let us write~$ \abs{P,P'} $ for this order.

Let $ \cA $ be the union of all the $ \cA_{P,P'} $. We wish to show that $ \family{\cA_i\mid i\in I} $ thinly splinters, using as the relation $\sim$ on~$ \cA $ the usual nestedness of separations. We shall prove first that~\cref{property:fin_cn} is satisfied, i.e. that each separation in an $ \cA_{P,P'} $ crosses only finitely many other separations from sets~$ \cA_{Q,Q'} $ with~$ \abs{Q,Q'}\le\abs{P,P'} $.

Making use of the tightness of the separations in the $ \cA_{P,P'} $, \cref{property:fin_cn} will follow immediately from the following assertion:

\begin{PROP}\label{prop:loc_fin_cross}
	Let $ (A,B) $ be a separation that efficiently distinguishes some two regular profiles in $ G $. Then $ G $ has only finitely many tight separations of order at most $ \abs{(A,B)} $ that cross~$ (A,B) $.
\end{PROP}

We shall derive~\cref{prop:loc_fin_cross} from the following lemma about tight separations:

\begin{LEM}\label{lem:tight_separators}
	Let $ (A,B) $ and $ (A',B') $ be two tight separations of~$ G $. Then $ (A',B') $ is either nested with $ (A,B) $, or its separator $ A'\cap B' $ is a $ \sub $-minimal $ x $-$ y $-separator in $ G $ for some pair $ x,y $ of vertices from~$ (A\cap B)\cup N(A\cap B) $.
\end{LEM}

\begin{proof}
	Since $ (A',B') $ is tight each of $ A'\sm B' $ and $ B'\sm A' $ contains some tight component of~$ G-(A'\cap B') $. If~$ A\cap B $ meets all tight components of~$ G-(A'\cap B') $ then in particular~$ A\cap B $ meets these two components, say in~$ x $ and in~$ y $. But then~$ A'\cap B' $ is a $ \sub $-minimal $ x $-$ y $-separator with~$ x,y\in A\cap B $.
	
	Therefore we may assume that $ A\cap B $ misses some tight component $ C' $ of $ G-(A'\cap B') $. By switching their names if necessary we may assume that this component $ C' $ is contained in~$ A\sm B $. Since $ C'\sub A\sm B $ has no neighbours in $ B\sm A $ but has $ A'\cap B' $ as its neighbourhood we can infer that~$ (A'\cap B')\sub A $.
	
	Consider now a tight component $ C $ of~$ G-(A\cap B) $ that is contained in~$ B\sm A $. From~$ (A'\cap B')\sub A $ it follows that~$ C $ does not meet~$ A'\cap B' $ and is hence contained in either~$ A'\sm B' $ or~$ B'\sm A' $. By possibly switching the roles of $A'$ and $B'$ we may assume that~$ C\sub A'\sm B' $. As above we can conclude from the tightness of~$ C $ that~$ (A\cap B)\sub C $.
	
	It remains to check two cases. If~$ (B\sm A)\cap(B'\sm A') $ is empty we have~$ B\sub A' $ and~$ B'\sub A $, that is, that~$ (A',B') $ is nested with~$ (A,B) $. The other remaining case is that~$ (B\sm A)\cap(B'\sm A') $ is non-empty.
	
	In that case, since~$ G $ is connected, the set~$ (A\cap B)\cap(A'\cap B') $ must be non-empty as well, since $N((B\sm A)\cap(B'\sm A'))\subseteq (A\cap B)\cap(A'\cap B')$. Pick a vertex~$ z $ from that set. Since~$ (A',B') $ is tight~$ z $ has neighbours~$ x $ and~$ y $ in some tight components of~$ G-(A'\cap B') $ contained in~$ A'\sm B' $ and in~$ B'\sm A' $, respectively. Then~$ A'\cap B' $ is a $ \sub $-minimal $ x $-$ y $-separator in~$ G $, and moreover~$ x,y\in (A\cap B)\cup N(A\cap B) $ since~$ z\in A\cap B $.
\end{proof}


Let us now use~\cref{lem:tight_separators} to establish~\cref{prop:loc_fin_cross}:

\begin{proof}[Proof of~\cref{prop:loc_fin_cross}.]
	Since $ G $ is locally finite the set $ (A\cap B)\cup N(A\cap B) $ is finite. Therefore, by~\cref{lem:finitelymany_seps}, there are only finitely many $ \sub $-minimal $ x $-$ y $-separators of size at most $ \abs{(A,B)} $ with~$ x,y\in (A\cap B)\cup N(A\cap B) $. Leveraging again the fact that $ G $ is locally finite and using that $G$ is connected, we get that there are only finitely many separations of $ G $ with such a separator.
	
	The assertion now follows from~\cref{lem:tight_separators}  since we know by \cref{lem:distinguisher_tight} that $ (A,B) $ is tight. 
\end{proof}

The family $ \family{\cA_i\mid i\in I} $ therefore has~\cref{property:fin_cn}. With regard to~\cref{property:fish} it turns out that we already did the required work back in~\cref{sec:profinite_app}:

\begin{COR}\label{cor:property2}
	If $ \va_i\in\cA_i $ and $ \va_j\in\cA_j $ cross with $ \abs{i}<\abs{j} $, then $ \cA_j $ contains some corner separation of~$ a_i $ and~$ a_j $.
\end{COR}

\begin{proof}
	This is the assertion of~\cref{lem:findcorner1}.
\end{proof}

It remains to show that $ \family{\cA_i\mid i\in I} $ has~\cref{property:strong_submodular}. 

Using again our preparatory work from~\cref{sec:profinite_app}, we can now show that $ \family{\cA_i\mid i\in I} $ has~\cref{property:strong_submodular} using \cref{lem:crossnumb_graph}:

\begin{LEM}\label{lem:property3}
	If $ \va_i\in\cA_i $ and $ \va_j\in\cA_j $ cross with $ k=\abs{i}=\abs{j} $, then either $ \cA_i $ contains a corner separation of $ a_i $ and $ a_j $ with strictly lower $ k $-crossing number than~$ \va_i $, or else $ \cA_j $ contains a corner separation of $ a_i $ and $ a_j $ with strictly lower~$ k $-crossing number than~$ \va_j $.
\end{LEM}

\begin{proof}
	By switching their roles if necessary we may assume that the $ k $-crossing number of~$ \va_i $ is at most the $ k $-crossing number of~$ \va_j $.
	
	From~\cref{lem:findcorner2} it follows that $ \cA_j $ contains a corner separation of~$ a_i $ and~$ a_j $ whose opposite corner separation lies in either~$ \cA_i $ or~$ \cA_j $. Now~\cref{lem:crossnumb_graph} implies that the sum of the $ k $-crossing numbers of this pair of opposite corner separations is at most the sum of the~$ k $-crossing numbers of~$ \va_i $ and~$ \va_j $. This inequality is in fact strict since~$ \va_i $ and~$ \va_j $ cross each other but are each nested with both corner separations.
	
	If the first corner separation is not already as desired, that is, if its $ k $-crossing number is not strictly lower than the $ k $-crossing number of~$ \va_j $, we can infer that the $ k $-crossing number of the opposite corner separation is strictly lower than that of~$ \va_i $. Since we assumed in the beginning that the $ k $-crossing number of~$ \va_i $ is no greater than that of~$ \va_j $ this proves the claim.
\end{proof}

We are now ready to prove the main result of this subsection, which is similar to \cite{carmesin2020canonical}*{Theorem 7.5}:

\begin{THM}\label{thm:nested_loc_fin}
	Let $G$ be a locally finite connected graph and $\cP$ some set of robust regular profiles in~$G$. Then there exists a nested set $\cN$ of separations which efficiently distinguishes any two distinguishable profiles in~$\cP$. Moreover, this set is canonical, i.e. invariant under isomorphisms: If $\alpha:G\to G'$ is an isomorphism, then $\alpha(\cN(G,\cP))=\cN(\alpha(G),\alpha(P))$.
\end{THM}

\begin{proof}
	The combination of~\cref{prop:loc_fin_cross},~\cref{cor:property2}, and~\cref{lem:property3} shows that the family $ \family{\cA_i\mid i\in I} $ thinly splinters. The nested set $ N\sub\cA $ produced by~\cref{thm:thinly} meets each set $ \cA_i $ and thus disinguishes all pairs of dinstinguishable profiles in~$ \cP $ efficiently.
\end{proof}

The nested set found by~\cref{thm:nested_loc_fin} does not in general correspond to a tree-decomposition of~$ G $, as~\cref{ex:notd} demonstrated. However \cref{thm:nested_loc_fin} can be used to shown that for every fixed integer~$ k $ the subset of~$ N $ consisting of all separations of order at most~$ k $ gives rise to a tree-decomposition of~$ G $, as this subset will satisfy the conditions from \cref{lem:TreesetTD}. In particular we can use \cref{thm:nested_loc_fin} together with \cref {lem:TreesetTD} to prove \cite{carmesin2020canonical}*{Theorem 7.3}, that there is for every $ k\in\N $, every locally finite graph $G$ and every set $\cP$ of distinguishable robust regular profiles, pairwise distinguishable by a separation of order at most $k$, a canonical tree-decomposition of~$ G $ that efficiently distinguishes all profiles from $\cP$.

\subsection{Graphs with vertices of infinite degree}\label{sec:application_infinite}
When we consider graphs with vertices of infinite degree, the method of the previous section fails as we loose \cref{prop:loc_fin_cross}: It does not necessarily hold that every separation in an $\cA_{P,P'}$ crosses only finitely many other separations from sets $\cA_{Q,Q'}$ with $|Q,Q'| \leq |P,P'|$. Moreover, Dunwoody and Krön~\cite{DunwoodyKroen} gave an example of a graph which does not contain a canonical nested set of separations separating its ends. As ends induce robust regular profiles, in arbitrary graphs, it is not generally possible to find a canonical nested set of separations distinguishing all the robust regular profiles.

To show the result for locally finite graphs we made use of the observation that only finitely many different \emph{separators} are involved, and then used that every separator appears in only finitely many separations.
Thus in this section instead of applying \cref{thm:thinly} directly to some set of \emph{separations}, we are going to apply it to only the set of \emph{separators}.

With this approach we show that in an arbitrary graph you can find a canonical nested set of separators which efficiently distinguishes all the robust regular profiles in $G$. We shall make the meaning of this more precise shortly. We propose that this set of separators is a natural intermediate object for distinguishing profiles.
Moreover we will show that if we restrict ourselves to the set of robust principal profiles -- which we will define at the end of this section -- then from this set we can build both a non-canonical nested set of separations as in \cref{thm:Johannes} (from \cite{carmesinhalinconj}) as well as a canonical tree of tree-decompositions in the sense of~\cite{carmesin2020canonical}. 

Either of these objects can trivially be converted back to a set of separators. Our technique splits the process of building either of these cleanly into two independent steps, which makes it more accessible than the proofs in \cite{carmesinhalinconj}~and~\cite{carmesin2020canonical}. Moreover, the first step of this process also works for non-principal but regular profiles, allowing us to also get a (intermediate) result for those profiles, unlike the theorems from \cite{carmesinhalinconj} and \cite{carmesin2020canonical}. Note that distinguishing non-principal profiles is also discussed extensively in \cite{ElmKurkofkaInfTangles}.

Many of the techniques applied throughout are similar to or inspired by arguments made in~\cite{carmesin2020canonical}, particularly the approach of minimising the crossing-number, even though the different levels of abstraction make it hard to draw concrete parallels.

Let us now begin with the formal notation.
We say that \emph{a set of vertices $X \subseteq V(G)$ efficiently distinguishes} a pair $P$ and $P'$ of profiles in $G$ if there exists a separation $(A,B)$ of $G$ with separator $A\cap B = X$ which efficiently distinguishes $P$ and $P'$. Such a separation $(A,B)$ is then a \emph{witness} that $X$ efficiently distinguishes $P$ and $P'$.

Given some set of distinguishable robust regular profiles $\cP$ of an (infinite) graph $G$, we define as $\cA$ the set of all such separators $X$ which distinguish some pair of profiles in $\cP$ efficiently.
We say that a separator $X$ is \emph{nested} with $Y \in \cA$, i.e.~$X \sim Y$, whenever $X$ is contained in $C \cup Y$ for some component $C$ of $G-Y$. In other words $Y$ does not properly separate any two vertices of $X$.
This relation is reflexive, the following lemma shows that it is also symmetric on $\cA$.
Unfortunately, its natural extension to all finite subsets of $V(G)$ is not. The reader should take note that this will lead to some situations where we argue that some set $Y$ is nested with some $X \in \cA$ \emph{provided that~$Y \in\cA$}.

\begin{LEM}\label{lem:nested_sym}
If $X,Y\in \cA$ and $X$ is contained in $Y$ together with some component of $G-Y$, then $Y$ is contained in $X$ together with some component of $G-X$.
\end{LEM}
\begin{proof}
    Pick a separation $(A,B)$ witnessing that $X\in \cA$. Since this separation efficiently distinguishes two regular profiles, by \cref{lem:distinguisher_tight}, there are at least two tight components of $G-X$, one in either side of $(A,B)$.
At least one of these tight components, say $C$, does not meet $Y$ and is therefore contained in a connected component $C'$ of $G-Y$.
Now, as required, we find \[ X = N(C) \subseteq C\cup N(C) \subseteq C'\cup N(C')\subseteq C'\cup Y. \qedhere\]
\end{proof}

As usual we take as $I$ the set of pairs of distinguishable profiles in $\cP$.
But this time we define $\cA_{P,P'}$ for each pair $P, P'$ in $I$ to be the set of all the sets of vertices in $G$ which distinguish $P$ and $P'$ efficiently.
All these separators in $\cA_{P,P'}$ have the same size; this size shall be $|P,P'|$.

We claim that $\{\cA_{P,P'}\mid \{P,P'\}\in I\}$ thinly splinters.
Before we can show \cref{property:fin_cn} we need to make two basic observations about how the vertices of a crossing pair of separators in $\cA$ lie:

\begin{LEM}\label{lem:tight_meeting}
If $X,Y\in \cA$ cross, then $Y$ contains a vertex from every tight component of~$G-X$.
\end{LEM}
\begin{proof}
If $C$ is a tight component of $G-X$ such that $Y$ does not contain any vertex of~$C$ then $C$ is contained in some component $C'$ of $G-Y$.
However, then $X = N(C) \subseteq C'\cup Y$, i.e., $X$ is nested with $Y$ contradicting the assertion.
\end{proof}
\begin{LEM}\label{lem:crossing_min_sep}
If $X,Y\in \cA$ cross, then $Y$ contains a pair of vertices $v$ and $w$ such that $X$ is a $\subseteq$-minimal $v$--$w$-separator.
\end{LEM}
\begin{proof}
There are at least two tight components $C_1,C_2$ of $G-X$ and $Y$ meets both of them by \cref{lem:tight_meeting}. Let $v$ be a vertex in $Y\cap C_1$ and $w$ a vertex in $ Y\cap C_2$. As both $C_1$ an $C_2$ are tight components, $X$ is indeed a $\subseteq$-minimal $v$-$w$-separator.
\end{proof}

We can now combine these with \cref{lem:finitelymany_seps} to prove \cref{property:fin_cn}.

\begin{LEM}\label{lem:fin_cn}
	For every pair of profiles $P,P' \in \cP$ every $X \in \cA_{P,P'}$ has finite k-crossing-number for all $k \le |P,P'|$.
\end{LEM}
\begin{proof}
	By \cref{lem:crossing_min_sep}, for every $Y\in \cA$ of size $k$ which crosses $X$, there are vertices $v,w \in X$ which are minimally separated by $Y$.
	However, there is only a finite number of pairs of vertices $v,w$ in $X$ and by \cref{lem:finitelymany_seps} every pair has only finitely many minimal separators of size $k$.
Therefore only finitely many such $Y\in \cA$ exist.
\end{proof}

The following lemmas show how the separators of corner separations behave under our new nestedness relation.
We will need these to prove \cref{property:fish,property:strong_submodular}.
Recall from \cref{sec:thinly} that a \emph{corner} of two separators $X,Y \in \cA$ is a separator $Z \in \cA$ which crosses only elements of $\cA$ which cross either $X$ or~$Y$. Note that this does not imply that $Z$ is nested with $X$ and $Y$.

\begin{LEM}\label{lem:fish_gsep}
	Let $X,Y \in \cA$ be a crossing pair of separators and let $(A_X, B_X)$ and $(A_Y, B_Y)$, respectively, be separations which witness that these are in $\cA$.
        Then for every $Z \in \cA$ which is nested with both $X$ and $Y$ there is a component $C_Z$ of $ G - Z $, such that $X\cup Y \sub C_Z \cup Z$.
	In particular $(A_X \cup A_Y) \cap (B_X \cap B_Y)$ , the separator of $(A_X, B_X) \join (A_Y, B_Y)$, is a corner of $X$~and~$Y$ provided that it lies in $\cA$.
\end{LEM}

\begin{proof}
	We first show that $Z$ does not separate $X$ and $Y$.
    Since $Z$ is nested with $X$ and $X$ efficiently distinguishes two regular profiles there is, by \cref{lem:distinguisher_tight}, a tight component $C_X$ of $G-X$ which is disjoint from $Z$.
	By \cref{lem:tight_meeting}, there is a vertex $y \in C_X \cap Y \sub Y \sm Z$.

	By a symmetrical argument there also exists a vertex $x \in X \sm Z$. Since $C_X$ is tight there is a path from $x$ to $y$ contained in $C_X$ except for $x$. This path avoids $Z$.

	Now, since $Z$ is nested with $X$ there is a component $C_Z$ of $G-Z$ which contains $X \sm Z$.
In particular this component contains $x$.
Similarly, there is a component of $G-Z$ containing $Y \sm Z$ and hence, in particular,~$y$.
Since $Z$ does not separate $x$ and $y$ this component is the same as $C_Z$.
Therefore $X\cup Y\subseteq C_Z\cup Z$, as required. In particular, if $(A_X \cup A_Y) \cap (B_X \cap B_Y)\in \cA$ then $(A_X \cup A_Y) \cap (B_X \cap B_Y)\subseteq C_Z\cup Z$, hence $(A_X \cup A_Y) \cap (B_X \cap B_Y)\sim Z$ and  therefore $(A_X \cup A_Y) \cap (B_X \cap B_Y)$ is a corner of~$X$ and~$Y$.
\end{proof}

\begin{LEM}\label{lem:submodular}
	Let $X,Y \in \cA$ be a crossing pair of separators and let $(A_X, B_X)$ and $(A_Y, B_Y)$, respectively, be witnesses that these are in $\cA$.
If $Z \in \cA$ is nested with $X$, and each of the corner separations $(A_X,B_X)\join(A_Y,B_Y)$ and $(A_X,B_X)\meet (A_Y,B_Y)$ distinguishes some pair of profiles efficiently then $Z$ is nested with one of the separators $(A_X\cup A_Y)\cap (B_X\cap B_Y)$ or $(A_X\cap A_Y) \cap (B_X\cup B_Y)$.
\end{LEM}
\begin{proof}
Since $Z$ and $X$ are nested there is a component $C^Z$ of $G-X$ such that $Z\subseteq C^Z\cup X$.
Let us assume without loss of generality that $C^Z\subseteq A_X$, we will show that $Z$ is nested with $(A_X\cup A_Y)\cap(B_X\cap B_Y)$.

Since $(A_X,B_X)\join(A_Y,B_Y)$ efficiently distinguishes some regular profiles there is, by \cref{lem:distinguisher_tight}, a tight component of $(A_X\cup A_Y)\cap (B_X\cap B_Y)$ contained in $(B_X\cap B_Y)$. 
However $Z\subseteq A_X$, so this component cannot meet $Z$.
Hence, by \cref{lem:tight_meeting}, $Z$ cannot cross the separator~$(A_X\cup A_Y)\cap(B_X\cap B_Y)$.
\end{proof}
These now allow us to reuse \cref{lem:findcorner1,lem:findcorner2} to prove \cref{property:fish,property:strong_submodular}:

\begin{LEM}\label{lem:fish_exists}
If two separators $X \in \cA_{P,P'}$ and $Y \in \cA_{Q,Q'}$ cross and $|P,P'| < |Q,Q'|$
then there is a corner $Y' \in \cA_{Q,Q'}$ of $X$ and $Y$ which is nested with $X$.
\end{LEM}
\begin{proof}
 Let $(A_X,B_X)$ be a separation witnessing that $X\in \cA_{P,P'}$ and let $(A_Y,B_Y)$ be a separation witnessing that $Y\in \cA_{Q,Q'}$.
By \cref{lem:findcorner1} there is a corner separation of $(A_X,B_X)$ and $(A_Y,B_Y)$ which also distinguishes $Q$ and $Q'$ efficiently.
The separator $Y'$ of this corner separation does not meet all tight components of $G-X$, so $Y'$ is nested with $X$ and thus is by \cref{lem:fish_gsep} as desired.
\end{proof}

\begin{LEM}\label{lem:strong_submodular}
If two separators $X \in \cA_{P,P'}$ and $Y \in \cA_{Q,Q'}$ cross and $|P,P'| = |Q,Q'| = k$
then either there is a corner $Y' \in \cA_{Q,Q'}$  of $X$ and $
Y$ which has a strictly lower $k$-crossing-number than $Y$,
or there is a corner $X' \in \cA_{P,P'}$ of $X$ and $Y$ which has strictly lower $k$-crossing-number than $X$.
\end{LEM}
\begin{proof}
	By switching their roles if necessary we may assume that the $k$-crossing number of~$ Y$ is at most the $ k$-crossing number of~$ X $.  Let $(A_X,B_X)$ be a separation witnessing that $X\in \cA_{P,P'}$ and let $(A_Y,B_Y)$ be a separation witnessing that $Y\in \cA_{Q,Q'}$. By \cref{lem:findcorner2} there is a corner separation of $(A_X,B_X)$ and $(A_Y,B_Y)$ which efficiently distinguishes $P$ and $P'$ and whose opposite corner separation efficiently distinguishes either $P$ and $P'$ or $Q$ and $Q'$. Let us denote their separators as $Z$ and $Z'$ respectively.	
	
	By \cref{lem:fish_gsep,lem:submodular} and the fact that $Z$ and $Z'$ are nested with both $X$ and $Y$ we have that $Z$ and $Z'$ are corners of $X$ and $Y$ and that the sum of the $k$-crossing numbers of $Z$ and $Z'$ is strictly lower than the sum of the $k$-crossing numbers of $X$ and~$Y$.
	
	Thus, if the $k$-crossing number of $Z$ is strictly lower than the $ k$-crossing number of~$ X $, we can take $Z$ for $X'$. Otherwise we can infer that the $k $-crossing number of $Z'$ is strictly lower than that of~$Y$. Since we assumed in the beginning that the $ \abs{i} $-crossing number of~$Y$ is not greater than that of~$X $. This proves the claim since we can then take $Z'$ for $X'$ or $Y'$, depending.
\end{proof}

With this all the requirements of \cref{thm:thinly} are satisfied.
Immediately we obtain the main result of this section:

\thmNestedSeparators* \vspace{-.6\baselineskip}\hfill\qedsymbol

As noted before, to be able to deduce \cref{thm:Johannes} and \cite{carmesin2020canonical}*{Remark 8.3} we restrict our set $\cP$ to be a set of \emph{principal} robust profiles.
A $k$-profile $P$ in $G$ is \emph{principal} if it contains for every set $X$ of less than $k$ vertices a separation of the form $(V(G) \sm C, C \cup X)$ where $C$ is a connected component of $G-X$.
In particular, every principal profile is regular.
Note that this notion of principal profiles is equivalent to the notion of `profiles' in Carmesin's \cite{carmesinhalinconj}; the term \emph{principal profiles} comes from \cite{carmesin2020canonical}. Observe that in locally finite graphs an inductive application of the profile property \ref{property:P} shows that every profile is principal.

This restriction to principal profiles is necessary for \cref{thm:Johannes}, as Elm and Kurkofka~\cite{ElmKurkofkaInfTangles}*{Corollary~3.4} have shown that there is a graph together with a set of (non-principal but robust and distinguishable) profiles, which do not permit the existence of a nested set of separations distinguishing all of them.

\subsubsection{Nested sets of separations}\label{sub:johannes}
If we restrict $\cP$ to a set of principal profiles, the nested set of separators from \cref{cor:nested_separators} can be transformed into a nested set of \emph{separations} which still distinguishes all the profiles in $\cP$ if we give up on canonicity.
This task is not entirely trivial.

The natural approach would be to take for each separator every one of the separations belonging to one of its tight components, i.e. the separation $(C\cup X, V\sm C)$ for every tight component $C$ of $G-X$.
However, if the separators overlap the resulting set of separations might not be nested.
The following  recent result by Elm and Kurkofka states that we need to omit no more than one of the tight components for each separator to reclaim nestedness.
\begin{THM}[{\cite{ElmKurkofkaInfTangles}*{Corollary 6.1}}]\label{thm:j+a}
Suppose that $\cY$ is a principal collection of vertex sets in a connected graph $G$. Then there is a function $\cK$ assigning to each $X\in\cY$ a subset $\cK(X)\subseteq \cC_X$ (the set $\cC_X$ consists of the components of $G - X$ whose neighbourhoods are precisely equal to $X$) that misses at most one component from $\cC_X$, such that the collection \[\left\{\left\{V\sm K,X\cup K\right\}\mid X\in \cY\text{ and }K\in \cK(X)\right\}\]
is nested.
\end{THM}
Here, a principal collection of vertex sets is just a set $\cY$ of subsets of $V$ such that, for every $X,Y\in \cY$, there is at most one component of $G-X$ which is met by $Y$. In particular, any nested set of separators is a principal collection of vertex sets.

Having for every separator all but one of these tight component separations is still enough to efficiently distinguish all the profiles in $\cP$. However, as \cref{thm:j+a} does not give as a canonical choice for the function $\cK$, we need to give up the canonicity at this point. However, this still allows us to prove the following theorem by Carmesin:
\thmJohannes*
\begin{proof}If $G$ is not connected, then every robust principal profile of $G$ induces a robust principal profile on exactly one of the connected components of $G$. It is easy to see that we can then apply the theorem to all connected components from $G$ independently and obtain our desired nested set of separations of $G$ from those of the connected components together with separations of the form $(C,V\sm C)$ for connected components $C$ of $G$. Thus let us suppose that $G$ is connected.

    Let $N$ be the nested set of separations obtained by applying \cref{thm:j+a} to the set $\cN$ of separators obtained from \cref{cor:nested_separators}. Given any two profiles $P,Q \in \cP$ there is a separator $X$ in $\cN$ which efficiently distinguishes $P$ and $Q$. By \cref{lem:distinguisher_tight} there are two distinct tight components $C$ and $C'$ of $G-X$ such that both $(V\sm C, C\cup X)\in P$ and $(C'\cup X,V\sm C')\in P$ efficiently distinguish $P$ and $Q$. However, at least one of these two separations is an element of $N$.
\end{proof}

For the readers convenience, we also offer a direct proof of \cref{thm:Johannes} which does not use \cref{thm:j+a}. Instead we perform an argument akin to one of the arguments used in the proof of \cref{thm:j+a} but in slightly simpler form, as the statement we need is a weaker one than \cref{thm:j+a}.
\begin{proof}[Direct proof of \cref{thm:Johannes}] 
	Let $\cN$ be the nested set of separators obtained from \cref{cor:nested_separators}  applied to the set of robust principal profiles. Pick an enumeration of $\cN$ which is increasing in the size of the separators, i.e., an enumeration $\cN=\{X_\alpha \mid \alpha < \beta\}$ such that $|X_\alpha|\le |X_\gamma|$ whenever $\alpha < \gamma$.

    We will construct a transfinite ascending sequence of nested sets $(N_\gamma)_{\gamma \le \beta}$, of separations. Each $N_\gamma$ will contain only separations with separators in $\{X_\alpha \mid \alpha<\gamma\}$, and every pair of profiles efficiently distinguished by such a separator $X_\alpha$, $\alpha < \gamma$, will also be efficiently distinguished by some separation in $N_\gamma$.

For the successor steps of our construction suppose that we already constructed $N_\gamma$ and consider $X_\gamma$. Since $X_\gamma$ is nested with all $X_\alpha$ satisfying $\alpha < \gamma$ we know that $X_\gamma$ induces a consistent orientation of $N_\gamma$ since any separation $(A,B)\in N_\gamma$ satisfies either $X_\gamma \subseteq A$ or $X_\gamma \subseteq B$ but not both, as $|(A,B)|\le |X_\gamma|$.

\newcommand{\cD}{\mathcal{D}}
Consider the set $\cC$ of tight components of $G-X_\gamma$ and
let $\cD$ be the set of the remaining, non-tight, components of $G-X_\gamma$.

Given any separation $(A,B)\in N_\gamma$ pointing away from $X_\gamma$ (that is $X_\gamma \subseteq A$), the side $B$ is contained in the union of one component $C_B\in \cC$ together with some components in~$\cD$: Since $X_\gamma$ is nested with $A\cap B$ there is a component in $G-X_\gamma$ containing $(A\cap B) \sm X_\gamma$, thus, any other component $C$ of $G-X_\gamma$ meeting $B$ does not meet $A\cap B$ and must therefore satisfy $N(C)\subseteq A\cap B\cap X_\gamma$, i.e. this component is not tight.

Given a tight component $C\in \cC$ let $\cD_{C} \subseteq \cD$ be the set of all components $D$ in $\cD$ with the property that there is some $(A,B)\in N_\gamma$ pointing away from $X_\gamma$ such that $D$ meets $B$ and $C_B=C$.
Informally, these sets $\cD_{C}$ are the components which we will need to group together with their $C$ when choosing our next separations.
The $\cD_{C}$ are pairwise disjoint: Indeed, given two separations $(A,B)$ and $(A',B')$ pointing away from $X_\gamma$, if $(B',A')\le (A,B)$ then the set $B'$ and $B$ are disjoint, and if $(A,B)\le (A',B')$, then $(A'\cap B')\sm X_\gamma$ and $(A\cap B) \sm X_\gamma$ cannot be contained in different tight components of $G-X_\gamma$.

Let $N_{\gamma^+}$ consist of $N_\gamma$ together with, for every tight component $C\in\cC$ of $G-X_\gamma$, the separation $\Big(C\cup \bigcup \cD_{C}\cup X_\gamma,\, V(G) \sm \left(C \cup \bigcup\cD_{C}\right) \Big)$. It is easy to see that this set is a nested set of separations. Moreover, any pair of profiles efficiently distinguished by $X_\gamma$ is efficiently distinguished by one of these new separations.

For limit ordinals $\gamma$ let $N_\gamma \coloneqq\bigcup_{\alpha<\gamma}N_\alpha$, this set is nested since every pair in $N_\gamma$ is already in some $N_\alpha$.

Then $N \coloneqq N_\beta$ is the desired nested set of separations.
\end{proof}

\subsubsection{Canonical trees of tree-decompositions}\label{sub:totd}
To canonically and efficiently distinguish a robust set of principal profiles in a graph Carmesin, Hamann and Miraftab~\cite{carmesin2020canonical} introduced more complex objects than nested sets of separations: trees of tree-decompositions.
These consist of a rooted tree where every node is associated with a tree-decomposition. At the root this is a tree decomposition of $G$. At every remaining node there is a tree-decomposition of one of the torsos of the tree-decomposition at the parent node.
Their main result is the following:
\thmJMB*

We can also construct such a tree of tree-decompositions from our nested set of separators.
In order to do that, let us recall the most important definitions from~\cite{carmesin2020canonical}.

In a rooted tree $(T,r)$, the \emph{level} of a vertex $t\in V(T)$ is $d(t,r)+1$.
A \emph{tree of tree-decompositions} is a triple $((T, r), (G_t)_{t\in V(T)}, (T_t , \mathcal{V}_t)_{t\in V(T)})$ consisting of a rooted tree $(T, r)$, a family $(G_t)_{t\in V(T)}$ of graphs and a family $(T_t,\mathcal{V}_t)_{t\in V(T)}$ of tree-decompositions of the~$G_t$. The graphs $G_t'$ assigned to the neighbours $t'$ on the next level from a node $t\in V(T)$ shall be distinct torsos of the tree-decomposition $(T_t,\mathcal{V}_t )$.
This tree of tree-decompositions is a tree of tree-decompositions \emph{of $G$}, if $G_r=G$.

A separation $(A,B)$ of $G$ \emph{induces} a separation $(A',B')$ of $G_t$ if $A\cap G_t=A'$ and $B\cap G_t=B'$.
Given two profiles $P,P'$, we say that a tree of tree-decompositions (\emph{efficiently}) \emph{distinguishes} $P$ and $P'$ if there is a separation $(A,B)$ in $G$ (efficiently) distinguishing them and a node $t\in V(T)$ such that the separation induced by $(A,B)$ on~$G_t$ is one of the separation induced by the tree-decomposition $(T_t,\mathcal{V}_t)$ of~$G_t$.

In order to deduce \cref{thm:jmb} from \cref{cor:nested_separators} it is useful to observe that our set of separators is nested in an even stronger sense: We say that two separators $X$ and $Y$ are \emph{strongly nested} if there is a component $C$ of $G-X$ such that~${Y\subseteq C\cup N(C)}$ and there is a component $C'$ of $G-Y$ such that~${X\subseteq C'\cup N(C')}$. The separators from the nested set $\cN$ from \cref{cor:nested_separators} are strongly nested:
\begin{LEM}\label{lem:strong_nested}
 If $X$ and $Y$ are a pair of nested separators each of which efficiently distinguishes some pair of robust principal profiles, then they are strongly nested.
\end{LEM}
\begin{proof} We show that there is a component $C$ of $G-X$ such that~${Y\subseteq C\cup N(C)}$.

If $Y\subseteq X$ the statement is obvious, by picking as $C$ a tight component of $G-X$.
So we may assume that $ Y $ meets some component $C$ of $G-X$ in a vertex $v \in Y \cap C$.
By nestedness $Y \subseteq C \cup X$.
Suppose for a contradiction that $Y \not\subseteq C \cup N(C)$, i.e., $Y$ contains a vertex $w\in X\sm N(C)$.

Since $Y$ efficiently distinguishes two principal profiles there are two distinct tight components $C_1,C_2$ of $G-Y$, by \cref{lem:distinguisher_tight}.
$X$ meets at most one of $C_1$ and $C_2$ since it is nested with $Y$; without loss of generality we may assume $X\cap C_2=\emptyset$.
Since $C_2$ is a tight component of $G-Y$ there is a path $P$ from $v$ to $w$ with all its interior vertices in $C_2$.
On the other hand $v$ lies in $C$ and $w$ outside of $C \cup N(C)$, so $N(C)$ separates $v$ from $w$. But $N(C) \sub X$ does not meet $P$ since $X\cap C_2=\emptyset$.
This is a contradiction.
\end{proof}
Note that for a separator $X$ to be strongly nested with itself is a non-trivial property: It is precisely the statement that there is a tight component of $G-X$. Thus, if we talk about a strongly nested set of separators, we mean that not only any pair of distinct separators from that set is strongly nested, we also require each of the separators from that set to be nested with itself.

Next we show that we can close our strongly nested set under taking subsets:
\begin{LEM}
 Let $\cN$ be a strongly nested set of separators and let $\cN'$ be the set of all subsets of elements of $\cN$. Then $\cN'$ is strongly nested as well.
\end{LEM}
\begin{proof} Let $X, Y\in \cN$ and let $X'\subseteq X, Y' \subseteq Y$, possibly equal.
    Take $C_X$ to be a component of $G - X$ for which $Y \subseteq C_X \cup N(C_X)$, then in particular $Y' \subseteq C_X \cup N(C_X)$.
    Since $X' \subseteq X$ there is some component $C_{X'} \supseteq C_X$ of $G-X'$, thus $Y' \subseteq C_{X'} \cup N(C_X')$.

    By symmetry we also find a component $C_{Y'}$ such that $X' \subseteq C_{Y'} \cup N(C_{Y'})$
\end{proof}

So let $\cN'$ be the strongly nested set of all subsets of separators from $\cN$, the canonical nested set of separators from \cref{cor:nested_separators}. As such, $\cN'$ is canonical as well.
The following lemma about separations with strongly nested separators will allow us to construct a tree of tree-decompositions from $\cN'$ inductively, starting with the separators of lowest size.

\begin{LEM}\label{lem:notnested_strongnested}
 If $X,Y$ are distinct strongly nested separators and $(A_X,B_X)$ and $(A_Y,B_Y)$ are separations with separators $X$ and $Y$ respectively, such that $Y\subseteq B_X$, $X\subseteq B_Y$, then either $(A_X,B_X)$ and $(A_Y,B_Y)$ are nested, or there is a component $C$ of $G-(X\cap Y)$ which meets neither $X$ nor $Y$.
\end{LEM}
\begin{proof}
 Suppose that $(A_Y,B_Y)\not\le (B_X,A_X)$. Then either there is a vertex in $A_Y$ which does not lie in $B_X$, or there is a vertex in $A_X$ which does not lie in $B_Y$. Since $A_X\cap B_X = X \subseteq B_Y$ and $A_Y\cap B_Y=Y\subseteq B_X$ either of these cases implies that there is a vertex $v$ in $(A_X\sm B_X)\cap (A_Y\sm B_Y)$. This vertex $v$ needs to lie in some component $C$ of $G-(X\cup Y)$. However, $C$ cannot send an edge to $X\sm Y$ since such an edge would contradict the fact that $(A_Y,B_Y)$ is a separation. Similarly, $C$ cannot be adjacent to any vertex of $Y\sm X$. Thus $C$ is in fact a component of $G-(X\cap Y)$ which meets either $X$ nor~$Y$.
\end{proof}

Now we are ready to deduce \cref{thm:jmb} from \cref{cor:nested_separators}:
\begin{proof}[Proof of \cref{thm:jmb}]
 Let $\cN'$ be as above.
 We will build our tree of tree-decompositions inductively level-by-level, adding at stage $k$ to every node $t$ on level $k-1$ new neighbours on level $k$, one for every torso of the tree-decompositions $(T_t,\cV_t)$.
 We do this in a way that ensures the following properties:
 \begin{enumerate}[label=(\roman*)]
 \item If $d(r,t)=k$ then every separation in $(T_t , \mathcal{V}_t)$ has order $k+1$.\label{item:sep_size}
 \item Every separator in $\cN'$ of size at least $k+2$ is contained in exactly one of the torsos of $(T_t , \cV_t)$ whenever $d(t,r)\le k$.\label{item:large_seps}
\item If $d(t,r)=k$, every torso of $(T_t, \cV_t)$ meets at most one component of $G-X$ for every $X \in \cN'$ of size $\leq k$ with $X \subseteq V(G_t)$.\label{item:small_seps}
\end{enumerate}

Our inductive construction goes as follows:
For $k=0$ we consider the set $S_1$ which consists of, for every separator $X$ of size $1$ in $\cN'$ and every component $C$ of $G-X$, the separation $(C \cup X, V(G) \sm C)$, unless $C$ is the only component of $G-X$.

Observe that $S_1$ is a nested set of separations: any two separations with the same separator are nested by construction and for separations with distinct separators $X$ and $Y$ the separators are disjoint, so  $G - (X\cap Y) = G$ is connected and \cref{lem:notnested_strongnested} gives that the separations are nested.

Moreover every $\omega$-chain $(A_1,B_1) < (A_2,B_2) < \dots$ in $S_1$ has $\bigcap_{i\in\N} B_i = \emptyset$: We may assume without loss of generality that no two of these separations have the same separator since $S_1$ has no $3$-chain of separations with the same separator. On the other hand a path from a vertex in $\bigcap_{i\in \N} B_i$ to $A_1$ (which has finite length) would need to meet all the infinitely many disjoint separators $A_i\cap B_i$. 

Since $S_1$ contains no small separations by construction it is a regular tree set.
Thus by \cref{lem:TreesetTD} it induces a canonical tree-decomposition $(T_r , \cV_r)$ of $G_r=G$.
We assign this tree-decomposition to the root of our tree of tree-decompositions and shall now verify \cref{item:sep_size,item:large_seps,item:small_seps}.

Observe that this decomposition satisfies \cref{item:sep_size,item:small_seps} as we only used separators of size $1$ and every torso of $(T_t, \cV_t)$ meets at most one component of $G-X$ for every $X \in \cN'$ of size $\leq 1$ with $X \subseteq V(G_t)$. Moreover, \cref{item:large_seps} is also satisfied since every separator $X$ in $\cN'$ of size at least $2$ is nested with each of the separators used in $(T_r , \cV_r)$: Such a separator cannot be contained in two distinct torsos since then a separation with separator in $\cN'$ would separate them. Conversely, there is a torso which contains $X$: Otherwise consider a torso $V_t$ that contains as much of $X$ as possible and another torso $V_t'$ which contains a vertex in $X\sm V_t$. Then one of the edges on the path between $t$ and $t'$ in $T$ again corresponds to a separation which separates $X$. But this is not possible since the separators of these separations are in $\cN'$ and thus nested with $X$.

For the $k$-th step of our construction, for $k\ge1$, we attach at every node $t$ on level $k-1$ of our so-far constructed tree of tree-decompositions, for every torso $G'$ of $(T_t,\cV_t)$ a new node $t'$ (which then is at level $k$) with $G_{t'} \coloneqq G'$.
We the independently construct tree-decompositions for each of these torsos $G_{t'}$. For every torso we use all those separators from $\cN'$ which are of size $k+1$ and lie inside that torso. Note that \cref{item:large_seps} guarantees that every separator in $\cN'$ of size $k+1$ is contained in exactly one of the newly added torsos.

Given one torso $G_{t'}$ of the tree-decomposition $(T_t,\cV_t)$, we let $S_{k+1}$ be the set of all separations $(A,B)$ of $G_{t'}$ of order $k+1$ with separator in $\cN'$ and the property that $A\sm B$ is a component of $G-(A\cap B)$ but not the only one.

We claim that $S_{k+1}$ is a nested set of separations. Indeed, if two separations from $S_{k+1}$ with different separators $X$ and $Y$ were to cross then by \cref{lem:notnested_strongnested} there would be a component of $G_{t'}-(X\cap Y)$ avoiding $X$ and $Y$. However, $X\cap Y$ has size less than $k$, lies in $\cN'$ and $G_{t'}$ meets, by \cref{item:small_seps}, at most one component of $G-(X\cap Y)$. Hence if we take vertices $x$ and $y$ in $G_{t'}-(X\cap Y)$ we find a path $P$ between them in $G-(X\cap Y)$. But since $G_{t'}$ is obtained from $G$ by repeatedly building a torso, $P\cap G_{t'}$ needs to contain a path between $x$ and $y$ in $G_{t'}$. In particular, this path does not meet $X\cap Y$ and thus $G_{t'}-(X\cap Y)$ has only one component, in particular every component of $G_{t'}-(X\cap Y)$ meets $X$ and $Y$. 

Now consider an $\omega$-chain $(A_1,B_1) < (A_2,B_2) < \dots$ in $S_{k+1}$. We may assume without loss of generality that no two of these separations have the same separator, as in the case $k=0$. If $\bigcap_{i\in \N} B_i$ is non-empty then its neighbourhood $Z \coloneqq N_{G_{t'}}(\bigcap_{i\in\N}B_i)$ needs to be properly contained in some $A_l\cap B_l$: Every vertex in $Z$ needs to be contained in some $A_m\cap B_m$ and if such a vertex lies in $A_m\cap B_m$, then it also lies in $A_n\cap B_n$ for every $n\ge m$. In particular, if $|Z|\ge k+1$, there would be an $m$ such that $A_m\cap B_m\subseteq Z$ and thus $A_n\cap B_n=A_m\cap B_m\forall n\ge m$ contradicting the assumption that no two of the $(A_l,B_l)$ have the same separator. Hence $|Z|\le k$ and we can easily find an $l$ such that $Z\subsetneq A_l\cap B_l$.

But then again $G_{t'}$ would meet two distinct components of $G-Z$: one meeting  $\bigcap_{i\in \N} B_i$ and one meeting $A_l$. This however is not possible since $|Z|<|A_l\cap B_l|$ and $Z\in \cN'$.

By construction $S_{k+1}$ contains no small separations, it is thus a regular tree set,
so by \cref{lem:TreesetTD} the set $S_{k+1}$ induces a canonical tree-decomposition $(T_{t'},\mathcal{V}_{t'})$ of $G_{t'}$.
In this way we construct all the tree-decompositions for nodes at level $k$. We need to verify \cref{item:sep_size,item:large_seps,item:small_seps}.
\cref{item:sep_size} is obvious. For \cref{item:large_seps} we observe that every separator in $\cN'$ of size at least $k+2$ which was contained in $G_{t'}$ was nested with every separator of a separation in $S_{k+1}$ and is therefore contained in exactly one of the torsos of $(T_{t'},\mathcal{V}_{t'})$, by the same argument as in the case $k=0$.

For \cref{item:small_seps} we note that for separators $X$ of size $\le k$ every torso of  $(T_{t'}, \cV_{t'})$ meets at most one component of $G-X$ as, by induction $G_{t'}$ itself only meets one component of $G-X$.
For a separator $X$ of size $k+1$ let $H$ be a torso of $(T_{t'}, \cV_{t'})$.
Firstly, $H$ meets at most one component of $G_{t'} - X$ since if $G_{t'}-X$ has more than one component then $X$ is one of the separators of $(T_{t'}, \cV_{t'})$ and therefore, as $S_{k+1}$ includes every separation of the form $(C\cup X,G_{t'}\sm X)$ for any component $C$ of $G_{t'}-X$, there needs to be a component $C$ of $G_{t'}-X$ such that $H$ is contained in $C\cup X$ .

Secondly, when building the torso $G_{t'}$ from $G$ we never add edges between distinct components of $G-X$ since we only add edges inside of separators in $\cN'$, which are nested with $X$. Hence, if $H$ would meet two components of $G-X$ it would also meet two component of $G_{t'}-X$. Hence $H$ meets at most one component of $G - X$. This gives~\cref{item:small_seps}.

\paragraph{Correctness}
Let us now verify that the so constructed tree of tree-decompositions $((T, r), (G_t)_{t\in V(t)}, (T_t , \mathcal{V}_t)_{t\in V(T)})$ -- which is canonical by construction -- has the properties \cref{i:jmb1,i:jmb2,i:jmb3} from the assertion.
The properties \cref{i:jmb2} and \cref{i:jmb3} are fulfilled by construction, so we only need to verify \cref{i:jmb1}.

Let $P,P'$ be two robust principal profiles from $\cP$. By \cref{cor:nested_separators}, $\cN'$ contains some separator $X$ which belongs to a separation efficiently distinguishing $P$ and $P'$, say $|X|=k$. By our inductive construction, there is a unique $G_t$ at level $k$ which contains $X$. As $P$ and $P'$ are principal profiles, there are two distinct components $C,C'$ of $G-X$ such that $(V(G)\sm C, C\cup X)\in P$, and $(V(G)\sm C', C'\cup X)\in P'$. We claim that $C \cap V(G_t)$ is not empty.

Note that $G_t$ is obtained from $G$ by repeatedly taking some separation $(A,B)$ of order $<k$ with $X \sub B$, deleting $A\sm B$ and making $A \cap B$ complete.
If we apply this operation for a single $(A,B)$ which, say, turns some graph $H$ with $V(H) \subseteq V(G)$ into $H'$ then this preserves for $H'$ the properties of $H$ that (i) $H[C \cap V(H)]$ is connected and (ii) every vertex in $X$ has, in $H$, a neighbour in $C \cap V(H)$. 
Thus every vertex in $X$ has, in $G_t$, a neighbour in $C \cap V(G_t)$ proving that $C \cap V(G_t)$ is non-empty.

By a symmetrical argument not only $C$ but also $C'$ meets some component of $G_t - X$.
Moreover, no two distinct components of $G-X$ can meet the same component of $G_t-X$: This would require an edge between these components, which would have to be added by the torso operation -- but this operation only adds edges inside a separator $Y\in \cN'$. And since $Y$ is nested with $X$, that is $Y$ meets only one component of $G-X$, this cannot add edges between different components of $G-X$.

Thus there is exactly one component $C_t$ of $G_t-X$ such that $C_t\subseteq C$ and this component is not the only one from $G_t-X$. So, by construction the separation $(C_t\cup X, G_t\sm X)$, which efficiently distinguishes the induced profiles of $P$ and $P'$ onto $G_t$ is induced by~$(T_t , \mathcal{V}_t)$.
\end{proof}

\bibliography{collective}

\begin{bibdiv}
\begin{biblist}

\bib{confing}{article}{
      author={Carmesin, J.},
      author={Diestel, R.},
      author={Hundertmark, F.},
      author={Stein, M.},
       title={Connectivity and tree structure in finite graphs},
        date={2014},
     journal={Combinatorica},
      volume={34},
      number={1},
       pages={1\ndash 35},
         DOI={10.1007/s00493-014-2898-5},
}

\bib{carmesin2020canonical}{article}{
      author={Carmesin, Johannes},
      author={Hamann, Matthias},
      author={Miraftab, Babak},
       title={Canonical trees of tree-decompositions},
      eprint={2002.12030},
}

\bib{carmesinhalinconj}{article}{
      author={Carmesin, Johannes},
       title={All graphs have tree-decompositions displaying their topological
  ends},
        date={2019},
        ISSN={1439-6912},
     journal={Combinatorica},
      volume={39},
      number={3},
       pages={545\ndash 596},
         DOI={10.1007/s00493-018-3572-0},
         url={https://doi.org/10.1007/s00493-018-3572-0},
}

\bib{ProfileDuality}{article}{
      author={Diestel, R.},
      author={Eberenz, Ph.},
      author={Erde, J.},
       title={Duality theorem for blocks and tangles in graphs},
        date={2017},
     journal={SIAM J.\ Discrete Math.},
      volume={31},
      number={3},
       pages={1514\ndash 1528},
         DOI={10.1137/16M1077763},
}

\bib{AbstractTangles}{article}{
      author={Diestel, R.},
      author={Erde, J.},
      author={Wei{\ss}auer, D.},
       title={Structural submodularity and tangles in abstract separation
  systems},
        date={2019},
     journal={J.~Combin.\ Theory (Series A)},
      volume={167C},
       pages={155\ndash 180},
         DOI={10.1016/j.jcta.2019.05.001},
}

\bib{ProfilesNew}{article}{
      author={Diestel, R.},
      author={Hundertmark, F.},
      author={Lemanczyk, S.},
       title={Profiles of separations: in graphs, matroids, and beyond},
        date={2019},
     journal={Combinatorica},
      volume={39},
      number={1},
       pages={37\ndash 75},
         DOI={10.1007/s00493-017-3595-y},
}

\bib{ProfiniteASS}{article}{
      author={Diestel, R.},
      author={Kneip, J.},
       title={Profinite separation systems},
        date={2020},
     journal={Order},
      volume={37},
       pages={179\ndash 205},
         DOI={10.1007/s11083-019-09499-y},
}

\bib{TangleTreeGraphsMatroids}{article}{
      author={Diestel, R.},
      author={Oum, S.},
       title={Tangle-tree duality in graphs, matroids and beyond},
        date={2019},
     journal={Combinatorica},
      volume={39},
       pages={879\ndash 910},
         DOI={10.1007/s00493-019-3798-5},
}

\bib{TangleTreeAbstract}{article}{
      author={Diestel, R.},
      author={Oum, S.},
       title={Tangle-tree duality in abstract separation systems},
        date={to appear},
     journal={Advances in Mathematics},
      eprint={1701.02509},
}

\bib{DiestelBook16noEE}{book}{
      author={Diestel, Reinhard},
       title={{Graph Theory}},
     edition={5},
      series={Graduate Texts in Mathematics},
   publisher={Springer},
        date={2017},
      volume={173},
        note={\doi{10.1007/978-3-662-53622-3}},
}

\bib{EndsAndTangles}{article}{
      author={Diestel, Reinhard},
       title={Ends and tangles},
        date={2017, Special issue in memory of Rudolf Halin},
     journal={Abh.\ Math.\ Sem.\ Univ.\ Hamburg},
      volume={87},
       pages={223\ndash 244},
         DOI={10.1007/s12188-016-0163-0},
}

\bib{AbstractSepSys}{article}{
      author={Diestel, Reinhard},
       title={Abstract separation systems},
        date={2018},
     journal={Order},
      volume={35},
       pages={157\ndash 170},
         DOI={10.1007/s11083-017-9424-5},
}

\bib{DunwoodyKroen}{article}{
      author={Dunwoody, M.~J.},
      author={Krön, B.},
       title={Vertex cuts},
        date={2015},
     journal={J.~Graph Theory},
      volume={80},
      number={2},
       pages={136\ndash 171},
         DOI={10.1002/jgt.21844},
}

\bib{ToTfromTTD}{article}{
      author={Elbracht, {Chr}istian},
      author={Kneip, Jakob},
      author={Teegen, Maximilian},
       title={Obtaining trees of tangles from tangle-tree duality},
      status={In preparation},
}

\bib{FiniteSplinters}{article}{
      author={Elbracht, {Chr}istian},
      author={Kneip, Jakob},
      author={Teegen, Maximilian},
       title={Trees of tangles in abstract separation systems},
      eprint={1909.09030},
}

\bib{ElmKurkofkaInfTangles}{article}{
      author={Elm, Ann-Kathrin},
      author={Kurkofka, Jan},
       title={A tree-of-tangles theorem for infinite-order tangles},
      eprint={2003.02535},
}

\bib{halin1991lattices}{article}{
      author={Halin, R},
       title={Lattices of cuts in graphs},
organization={Springer},
        date={1991},
     journal={Abh.\ Math.\ Sem.\ Univ.\ Hamburg},
      volume={61},
       pages={217\ndash 230},
         DOI={10.1007/BF02950767},
}

\bib{EndsAsTangles}{article}{
      author={Kneip, J.},
       title={Ends as tangles},
      eprint={1909.12628},
}

\bib{TreelikeSpaces}{article}{
      author={Kneip, J.},
      author={Gollin, P.},
       title={Representations of infinite tree sets},
        date={2020},
     journal={Order},
         DOI={10.1007/s11083-020-09529-0},
}

\bib{GMX}{article}{
      author={Robertson, N.},
      author={Seymour, P.D.},
       title={Graph minors. {X}. {O}bstructions to tree-decomposition},
        date={1991},
     journal={J.~Combin.\ Theory (Series B)},
      volume={52},
       pages={153\ndash 190},
         DOI={10.1016/0095-8956(91)90061-N},
}

\end{biblist}
\end{bibdiv}

\vspace{1cm}
\noindent
\begin{minipage}{\linewidth}
 \raggedright\small
   \textbf{Christian Elbracht},
   \texttt{christian.elbracht@uni-hamburg.de}

   \textbf{Jakob Kneip},
   \texttt{jakob.kneip@uni-hamburg.de}
   
   \textbf{Maximilian Teegen},
   \texttt{maximilian.teegen@uni-hamburg.de}

   Universit\"at Hamburg,
   Bundesstra\ss{}e 55,
   20146 Hamburg, Germany
\end{minipage}
\end{document}